\documentclass[11pt,a4paper,reqno]{amsart}
\usepackage{amssymb,amsmath,amsthm}
\usepackage[utf8]{inputenc}
\usepackage[T1]{fontenc}
\usepackage{enumerate}
\usepackage[all]{xy}
\usepackage{mathrsfs} 
\usepackage{fullpage}
\usepackage{comment}
\usepackage{xcolor}
\usepackage{stmaryrd}

\def\Z{\mathbb{Z}}

\def\N{\mathbb{N}}
\def\Q{\mathbb{Q}}

\def\cM{\mathcal{M}}
\def\cS{\mathcal{S}}

\def\cO{\mathcal{O}}
\def\sP{\mathscr{P}}
\def\sS{\mathscr{S}}
\def\fp{\mathfrak{p}}
\def\fd{\mathfrak{d}}
\def\fb{\mathfrak{b}}
\def\F{\mathbb{F}}
\def\div{\mathrm{div}}
\def\ord{\mathrm{ord}}

\newcommand{\Cc}{\mathbb{C}} 

\newcommand{\Rr}{\mathbb{R}}

\newcommand{\Zz}{\mathbb{Z}}

\newcommand{\Ff}{\mathbb{F}}

\newcommand{\irr}{{\rm Irred}(\mathcal{O}_K)}

\newcommand{\ul}[1]{\underline{#1}}

\renewcommand {\leq}{\leqslant}
\renewcommand {\geq}{\geqslant}
\renewcommand {\le}{\leqslant}
\renewcommand {\ge}{\geqslant}

\newtheorem{thm}{Theorem}[section]
\newtheorem{lem}[thm]{Lemma}
\newtheorem{prop}[thm]{Proposition}
\newtheorem{cor}[thm]{Corollary}

\theoremstyle{definition}
\newtheorem{defn}[thm]{Definition}

\newtheorem{rmk}[thm]{Remark}

\title{Integral Specializations of Polynomials: \hskip 90mm Hilbert Irreducibility, Coprimality, and Density}

\author{Angelot Behajaina}
\email{behajainaangelot@gmail.com}
\address{Department of Mathematics, Institute of Science Tokyo, 2-12-1 Ookayama, Meguro-ku, Tokyo 152-8551, Japan}

\author{Pierre D\`{e}bes}
\email{pierre.debes@univ-lille.fr}
\address{Univ. Lille, CNRS, UMR 8524, Laboratoire Paul Painlev\'{e}, F-59000 Lille, France}

\author{Martin Widmer}
\email{martin.widmer@tugraz.at}
\address{TU Graz, Institute of Analysis and Number Theory, Steyrergasse 30/II, 8010 Graz, Austria}
\begin{document}




\maketitle

\begin{abstract}
The Hilbert specialization property of a field, whereby irreducibility of polynomials is preserved under specialization of some of the variables, admits an integral counterpart where specialization takes place inside a ring and irreducibility is required over the ring. It is known that the ring $\mathcal{O}_K$ of integers of any global field $K$ satisfies this integrally Hilbertian property. Using a different approach, we reprove this result and strengthen it by showing that the corresponding integral Hilbert subsets of $\mathcal O_K^k$ are not only Zariski-dense, as required by integral hilbertianity, but also of positive density. 
A central ingredient is a generalization to the ring $\mathcal{O}_K$ of the Ekedahl--Poonen  formula originally proved for $\mathbb{Z}$ and $\F_q[t]$. This formula, of independent interest, gives the density of tuples in $\mathcal{O}_K^k$ for which coprime polynomials in $k$ variables and with coefficients in ${\mathcal O}_K$ take coprime values. Our generalization requires estimates for the number of points of bounded height in an arithmetic progression. We obtain such estimates as our third main result, thereby contributing to the theory of heights.
\end{abstract}

\section{Introduction} \label{sec:intro}
Inspired by Hilbert's irreducibility theorem, this paper studies irreducibility and coprimality questions for integral specializations of polynomials. We prove three main results, which are 
related in the context of this paper but rely on different techniques and may be of 
independent interest. 
Theorem \ref{thm:denposHilIrrThm}, which we start with below and was the original motivation, belongs to Hilbertian ring theory. Theorem \ref{thm:ekpoongen} extends to rings ${\mathcal O}_K$ of integers of global fields\footnote{A global field $K$ is a finite extension of $\Q$ or $\F_q(x)$ {with $K \cap \overline{\F_q}=\F_q$}, and ${\mathcal O}_K$ denotes the integral closure of $\mathbb{Z}$ or ${\mathbb{F}}_q[t]$ in $K$.} $K$ a formula due independently to Ekedahl and Poonen for ${\mathcal O}_K = {\mathbb Z}$ or  ${\mathcal O}_K = {\mathbb F}_q[t]$ on the density of points at which coprime polynomials assume coprime values. Theorem \ref{prop:crucari} provides some seemingly unknown though basic estimates for the number of points of bounded height in an arithmetic progression of the ring of integers of a global field.
\vskip 1mm 

\subsection{On integrally Hilbertian rings} Let ${\mathcal Z}$ be an integral domain with fraction field ${\mathcal Q}$. Let $\underline{T}=(T_1,\dots,T_k)$ and $\underline{Y}=(Y_1,\dots,Y_n)$ be two tuples of variables ($k, n \geq 1$). Let $P_1,\ldots,P_s \in \mathcal{Z}[\underline{T},\underline{Y}]$ be some polynomials, irreducible in ${\mathcal Q}[\underline{T},\underline{Y}]$ ($s\geq 1$), of degree $\geq 1$ in $\underline{Y}$ and such that the product $P=P_1 \cdots P_s$ has no \emph{fixed divisor \hbox{w.r.t.} $\underline{T}$}. Recall that by fixed divisor  \hbox{w.r.t.} $\underline{T}$, we mean a nonunit $a \in {\mathcal Z}$, $a\not=0$, such that $P(\underline{t},\underline{Y}) \equiv 0 \pmod{a}$ for every $\underline{t} \in {\mathcal Z}^k$. For example $2\in \Z$ is a fixed divisor of $P=(T^2-T)Y+(T^2-T-2)$ \hbox{w.r.t.} $T$. The following notion was introduced in \cite{BD25}.\footnote{following first investigations in \cite{BDKN22} under the name of \emph{Hilbert--Schinzel specialization property}.}

\begin{defn} \label{def:int-Hil}
The ring ${\mathcal Z}$ is said to be \emph{integrally Hilbertian} if for all $k,n,s, \underline T,\underline Y$ and $P_1,\ldots,P_s$ as above, the following subset is Zariski-dense:
\vskip 1mm

\centerline{${\mathcal H}_{\mathcal Z}(P_1,\ldots,P_s) = \left\{\underline{t} \in {\mathcal Z}^k \hskip 1pt \Bigg\vert \hskip 1pt 
\left\{\begin{matrix}
P_1(\underline{t},\underline{Y}),\dots,P_s(\underline{t},\underline{Y}) \hbox{\it \ are irreducible in } \mathcal{Q}[\underline{Y}], \hbox{\ \it and}\hfill \\
\hbox{\it the product } \prod_{i=1}^s P_i(\underline t,\underline Y) \hbox{\it  \ has no nonunit divisor in } {\mathcal Z}. \hfill \\
\end{matrix}
\right.
\right\}$.}
\end{defn}

\noindent
The double condition in the definition of ${\mathcal H}_{\mathcal Z}(P_1,\ldots,P_s)$ implies that 
\vskip 1mm

\noindent
(*) \emph{for all $\underline t\in {\mathcal H}_{\mathcal Z}(P_1,\ldots,P_s)$, the polynomials $P_1(\underline{t},\underline{Y}),\dots,P_s(\underline{t},\underline{Y})$ are {irreducible in $\mathcal{Z}[\underline{Y}]$}.}
\vskip 1,5mm
 
Compared to the classical Hilbertian property, the main feature of the {integrally} Hilbertian property is precisely that it yields specialized polynomials that are irreducible \emph{over the ring ${\mathcal Z}$} (and not merely over the field ${\mathcal Q}$). A typical consequence is a polynomial variant of the Schinzel Hypothesis for the ring ${\mathcal Z}[\underline T]$ (\cite{BD25}, \cite{BDN20b}).
\vskip 0,mm

It is known that, among other rings, the ring of integers ${\mathcal O}_K$ of any global field $K$ is integrally Hilbertian \cite[Corollary 1.4]{BD25}. In the special case ${\mathcal Z} = \Zz$, \cite[Theorem~1.7]{BD23} uses an alternate approach to show more precisely that the set ${\mathcal H}_{\Zz}(P_1,\ldots,P_s)$ is of positive density\footnote{The density used here on $\Zz^k$, and below on $({\mathcal O}_K)^k$, is precisely defined in \S \ref{ssec:density}.}. We extend this conclusion to all 
rings of  integers ${\mathcal O}_K$.

\begin{thm}\label{thm:denposHilIrrThm} Let ${\mathcal Z} ={\mathcal O}_K$ be the ring of integers of a global field $K$. Let $P_1,\ldots,P_s \in \mathcal{Z}[\underline{T},\underline{Y}]$ be some polynomials, irreducible in $K[\underline{T},\underline{Y}]$ ($s\geq 1$), of degree $\geq 1$ in $\underline{Y}$ and such that the product $P=P_1 \cdots P_s$ has no {fixed divisor \hbox{w.r.t.} $\underline{T}$}. Then the set ${\mathcal H}_{{\mathcal Z}}(P_1,\ldots,P_s)$ contains a subset of $({\mathcal O}_K)^k$ of positive density.
\end{thm}

As explained in the next subsections, some difficulties arise in the generalization, 
concerning certain ingredients of the approach.

\subsection{On the Ekedahl--Poonen formula} A central ingredient is a generalization of a formula due to Ekedahl and Poonen independently, which is interesting for its own sake.  Given some polynomials $Q_1,\dots,Q_\ell\in  \mathcal{O}_K[\underline{T}]$ ($\ell \geq 2$), coprime in $K[\underline{T}]$, define:
$$
\mathcal{R}=\{\underline{m} \in \mathcal{O}_K^k \mid Q_1(\underline{m}),\dots, Q_\ell(\underline{m})\,\,\hbox{\rm have no common nonunit divisor in }\,\mathcal{O}_K\}.
$$ 
For every nonzero ideal $\mathfrak{a}$ of $\mathcal{O}_K$, define
\begin{equation}\label{eq:defcp}
c_{\mathfrak{a}}=\#\{\underline{u} \in (\mathcal{O}_K/\mathfrak{a})^k \mid Q_1(\underline{u})= \dots = Q_\ell(\underline{u}) = 0\ \hbox{in } \mathcal{O}_K/\mathfrak{a}\}.
\end{equation} 

\noindent
For ${\mathcal Z}=\Zz$ and ${\mathcal Z}=\Ff_q[t]$, Ekedahl and Poonen proved the following formula for the density $\mu$ defined in \S \ref{ssec:density}; see \cite{Poo03} for $\ell=2$, and  
\cite{BD23} for $\ell\geq 2$ and ${\mathcal Z}=\Zz$:
\vskip 1mm

\centerline{$\displaystyle \mu(\mathcal{R})=\prod_{p\,\,\textrm{prime}} \left( 1-\frac{c_{p\mathbb{Z}}}{p^k} \right).
$}
\vskip1mm

We wish to extend this formula to the general case ${\mathcal Z}=\mathcal{O}_K$ and $\ell\geq 2$. Denote the set of prime ideals (\hbox{resp.} of nonzero prime ideals) of $\mathcal{O}_K$ by ${\rm Spec}\hskip 1pt \mathcal{O}_K$ (\hbox{resp.} by $({\rm Spec}\hskip 1pt \mathcal{O}_K)^\ast$), the (relative) Weil height on $K$ by $H_K$, the corresponding ``density'' on $\mathcal{O}_K^k$ by $\mu$, and the norm of an ideal $\mathfrak{p}\subset {\mathcal O}_K$ by ${\rm N}(\mathfrak{p})=\# (\mathcal{O}_K/\mathfrak{p})$; see Section \ref{sec:heights-statements} for details. 
Define the prime ideal version $\mathcal{R^\star}$ of $\mathcal{R}$:
\vskip 2mm
\centerline{$
\mathcal{R^\star}=\{\underline{m} \in \mathcal{O}_K^k \mid Q_1(\underline{m}),\dots, Q_\ell(\underline{m})\,\,\hbox{\rm have no common \emph{prime ideal} divisor in }\,\mathcal{O}_K\}
$.}

\begin{thm}\label{thm:ekpoongen} Let $K$ be a global field and $Q_1,\ldots,Q_\ell \in {\mathcal O}_K[\underline T]$ be $\ell\geq 2$ nonzero polynomials, coprime in $K[\underline{T}]$. Then we have: 
\vskip 1mm

\centerline{$\displaystyle 
\mu(\mathcal{R}^\star)=\prod_{\mathfrak{p}\in ({\rm Spec}\hskip 1pt \mathcal{O}_K)^\ast} \left( 1- \frac{c_\mathfrak{p}}{{\rm N}(\mathfrak{p})^{k}} \right).
$}
\end{thm}

We have $\mathcal{R^\star} \subset \mathcal{R}$, and $\mathcal{R^\star}=\mathcal{R}$ if $K$ is a global field of class number $1$. A problem in the class number $>1$ case is that we may have $\mu({\mathcal R}^\star)=0$
(if $c_{\mathfrak p} = {\rm N}(\mathfrak{p})^{k}$ for some prime ideal $\mathfrak p$), while we wish to show that $\mu({\mathcal R})>0$ if $c_{\mathfrak p} < {\rm N}(\mathfrak{p})^{k}$ for all \emph{principal} prime ideals $\mathfrak p$. We will prove a more general but more involved version of Theorem \ref{thm:ekpoongen}, namely Theorem \ref{thm:ekpoongen-2}, which allows to handle the set $\mathcal{R}\setminus \mathcal{R^\star}$ and leads to a positive density for the original set ${\mathcal R}$.

\vskip 2mm 

\subsection{On points of bounded height in arithmetic progressions} Our approach to Theorem \ref{thm:ekpoongen} rests on Poonen's method. Passing from the basic cases $K={\mathbb Q},{\mathbb F}_q(t)$ to a general  global field $K$ raises some difficulties. In particular we need estimates for the number of points of $\mathcal{O}_K$ of Weil height bounded by some number $B$ in an ``arithmetic progression''. What we call an arithmetic progression is a subset  ${\frak p} + a\subset \mathcal{O}_K$ with $a\in \mathcal{O}_K$ and ${\frak p}\subset \mathcal{O}_K$ a nonzero ideal. 

While this issue is easy in the basic cases, essentially reducing to counting integers of bounded absolute value for $K={\mathbb Q}$ and polynomials of bounded degree for $K={\mathbb F}_q(t)$, it appeared that, to our knowledge and despite its basic nature, 
no such estimate was available  for a general global field $K$. We address this gap. 

Our approach, which occupies all of Sections \ref{sec:heights-statements} and \ref{sec:crup}, uses height theory allied with geometry of numbers in the number field case and geometry of curves over finite fields in the function field case. The outcome, Theorem \ref{prop:crucari}, is an approximation formula for the desired number. The main terms are of the form 
\vskip 1mm

\centerline{$\displaystyle \frac{C B \log(B)^\beta}{{\rm N}(\frak p)}$}

\vskip 1mm

\noindent
with ${\rm N}(\frak p)$ the norm of $\frak p$ and $C, \beta$ some explicit constants depending on 
the global field $K$. The error terms are small as $N(\frak p)$ grows, uniformly in $a$. This comes for free in the number field case but requires additional effort in the function field  case.  

We refer to the statement of Theorem \ref{prop:crucari} for precise formulas.

\vskip 2mm
The rest of the paper has two main parts. The first one -- Section \ref{sec:heights-statements} and Section \ref{sec:crup} --, is devoted to our results on heights on global fields. The main result is Theorem \ref{prop:crucari}.  
It is stated in Section \ref{sec:heights-statements} and proved in Section \ref{sec:crup}. The second part is devoted to the Ekedahl--Poonen formula (Theorem \ref{thm:ekpoongen}) and the application to the integrally Hilbertian property of global fields (Theorem \ref{thm:denposHilIrrThm}). Section \ref{sec:Ekehdal-poonen} proves those results, except for an adjustement of a lemma of Poonen
that is postponed to Section \ref{sec:second-estimate}. Readers willing to admit Theorem \ref{prop:crucari}  and Corollary \ref{lem:boundinarithproBC} may go directly to the second part.

\section{Results on heights in global fields} \label{sec:heights-statements}
Let $K$ be a global field, that is, either a number field or a finite extension of $\F_q(t)$ such that $K \cap \overline{\F_q}=\F_q$. Denote by $K_0$ the field ${\mathbb Q}$ or ${\mathbb F}_q(t)$, and  by $\mathcal{O}_K$ the integral closure of ${\mathbb Z}$ or ${\mathbb{F}}_q[t]$ in $K$. Set $\fd=[K:K_0]$.

\subsection{Heights} \label{ssec:heights}
Let $\mathcal{M}_K$ be the set of places of $K$. 
For each $v\in \mathcal{M}_K$, let $v_0$ denote the place of $K_0$ below $v$, and let $K_v$ and ${K_0}_{v_0}$ be the corresponding completions. Set $\fd_v=[K_v:{K_0}_{v_0}]$, the local degree at $v$, and define the absolute value on $K$:
$$
|x|_v=
\begin{cases}
|\sigma_v(x)|&\textrm{if } $K$\textrm{ is a number field and }v\mid \infty,\\
{\rm N}(\fp)^{-\ord_v(x)/\fd_v}&\textrm{otherwise},
\end{cases}
$$
\noindent
For archimedean $v$, $\sigma_v:K \hookrightarrow \mathbb C$ denotes the associated embedding. For nonarchimedean $v$, $\fp$ is the associated prime ideal of $\cO_K$, and $\ord_v$ denotes the normalized valuation satisfying $\ord_v(K_v^\times)=\Z$.

We define the \emph{relative multiplicative Weil height} of  $x\in K$ by 
$$
H_K(x) = \prod_{v\in \mathcal{M}_K} \max\{1,|x|_v\}^{\fd_v},
$$
and the \emph{relative logarithmic Weil height} by 
$$
h_K(x) = \log_{q}(H_K(x))=\sum_{v\in \mathcal{M}_K} \fd_v \log_{q} \max(1,|x|_v),
$$
with the convention $q=e$ if $K_0=\mathbb{Q}$.
\subsection{Density} \label{ssec:density}
Let \(\mathcal{A}\subset \mathcal{O}_K^k\), where \(k\geq 1\). For \(B>0\) in the number field case, and \(B\in q^{\rho_K\mathbb{N}}\) in the function field case\footnote{We will see later that $h_K$ takes values in $\rho_K \N$ on $\cO_K$.} (with \(\rho_K\) as in \eqref{eq:defrho}), define

$$
\mathcal{A}_B
=
\left\{
\underline{m}\in\mathcal{A}
\;\middle|\;
\max_{1\leq i\leq k} H_K(m_i)\leq B
\right\}.
$$

Note that $\mathcal{O}_{K,B}^k:=((\mathcal{O}_K)^k)_B = ((\mathcal{O}_K)_B)^k$ and that $\mathcal{A}_B = \mathcal{A} \cap \mathcal{O}_{K,B}^k$.

\begin{defn}\label{defn:density}
The \emph{upper density} (resp., \emph{lower density}) of $\mathcal{A} \subset \mathcal{O}_K^k$ is given by the limit
$$
\mu_{{\rm sup}}(\mathcal{A})=\limsup_{B \rightarrow \infty}\frac{\# \mathcal{A}_B}{\#\mathcal{O}^k_{K,B}} \hskip 3mm ( \textrm{resp.},\, \mu_{{\rm inf}}(\mathcal{A})=\liminf_{B \rightarrow \infty}\frac{\# \mathcal{A}_B}{\#\mathcal{O}^k_{K,B}}).
$$ If $\mu_{{\rm sup}}(\mathcal{A})=\mu_{{\rm inf}}(\mathcal{A})$, the set $\mathcal{A}$ is said to \emph{have a density}, given by
$$
\mu(\mathcal{A})=\lim_{B \rightarrow \infty}\frac{\# \mathcal{A}_B}{\#\mathcal{O}^k_{K,B}}.  
$$
\end{defn}
\noindent
In the function field case, the limit $B \rightarrow \infty$ is understood to be taken over $B \in q^{\rho_K \N}$.

\subsection{Main result}
The main result of this section is the following estimate, which will be used in subsequent arguments. Its proof is postponed to Section \ref{sec:crup}. Let $\cS_\infty$ be the set of all infinite places in the number field case, and the set of all places above the place corresponding to the degree valuation in the function field case. When $K$ is a number field with signature $(r_1,r_2)$, we have $\#\cS_\infty=r_1+r_2$, and define
$$
E_{\cS_\infty}(B,\fp)=1 + \frac{1}{{\rm N}(\mathfrak{p})^{1/\fd}} B \log(B)^{\#\cS_\infty-2}.
$$
When $K$ is a function field, let 
\begin{equation}\label{eq:defrho}
\rho_K=\gcd\{\deg(v) \mid v \in \cS_\infty\},
\end{equation}
and define
$$
E_{\cS_\infty}(B,\fp)=
\begin{cases}
(\log_q(B))^{\#\cS_\infty}+ \frac{\deg(\fp)}{{\rm N}(\fp)^{1/\fd}}B(\log_q(B))^{\#\cS_\infty-2}& \textrm{if } \#\cS_\infty >1,\\
\log_q(B)+\frac{1}{{\rm N}(\fp)^{1/2}}\cdot B^{1/2}&\textrm{if } \#\cS_\infty=1.
\end{cases}
$$
Moreover, let
$$
\kappa_K=
\begin{cases}
\frac{2^{\#\cS_\infty}\pi^{r_2}}{(\#\cS_\infty-1)!|\Delta_K|^{1/2}}&\textrm{if }K\textrm{ is a number field},\\
\frac{\rho_K}{(\#\cS_\infty-1)! q^{g_K-1} \prod_{v \in \cS_\infty} \deg(v)} \left(\frac{q^{\rho_K}}{q^{\rho_K}-1}\right) \cdot \prod_{v \in \cS_\infty} \left( 1 - \frac{1}{q^{\deg(v)}}\right)&\textrm{if }K \textrm{ is a function field}.
\end{cases}
$$
\begin{thm}\label{prop:crucari} 
\label{lem:bounheigharithprog} Let $\mathfrak{p}$ be a nonzero ideal of $\mathcal{O}_K$ and let $a \in \mathcal{O}_K$. 
\vskip 1mm
\noindent 
(1) As $B \rightarrow \infty$, we have:
$$
\#(\mathfrak{p}+a)_B=  \frac{\kappa_K}{{\rm N}(\fp)}B(\log_q(B))^{\#\cS_\infty-1} + O_K( E_{\cS_\infty}(B,\mathfrak{p})). 
$$
\vskip 1mm
\noindent
(2) The subset $\mathfrak{p}+a \subset {\mathcal O}_K$ has density  
$$
\mu(\mathfrak{p}+a) = 1/{\rm N}(\mathfrak{p}).
$$
\end{thm}
\noindent
In the function field case (see Theorem \ref{cor:funcfiel}), we prove a more general and precise result (counting with ``exact height'') that is of independent interest, even in the simplest case $a=0$ and $\fp=\cO_{\cS}$, where $\cO_{\cS}$ denotes the ring of $\cS$-integers (see Definition \ref{defn:ringofint}).

As a consequence of Theorem \ref{lem:bounheigharithprog}, we obtain the following result, which will be used later. 
\begin{cor}\label{lem:boundinarithproBC} Assume $\fd\geq 2$ and let $0 <\varepsilon <1/\fd$. Define
$$
\theta_{\varepsilon}(B):=\log_q(B)^{\frac{1}{1-\varepsilon}}.
$$
Then there exists a function 
$$
D_{\varepsilon,B}=
o\left(\#\mathcal{O}_{K,B} \cdot \frac{\log_q(\theta_{\varepsilon}(B))}{\log_q(B)}\right), \quad \textrm{ as } B\rightarrow \infty,
$$ with the property that for all suitably large $B$, we have 
$$
\#(\fp+a)_B \leq D_{\varepsilon,B},
$$
for all nonzero ideal $\mathfrak{p} \subset \mathcal{O}_K$ with 
$
{\rm N}(\mathfrak{p})>\theta_{\varepsilon}(B)
$ and all $a \in \mathcal{O}_K$.
\end{cor}

\begin{proof} Since $\varepsilon < \frac{1}{\fd}$, we have 
\begin{equation}\label{eq:majexpN}
\frac{1}{{\rm N}(\fp)^{1/\fd}} \leq \frac{\deg(\fp)}{{\rm N}(\fp)^{1/\fd}} =O\left(\frac{1}{{\rm N}(\fp)^{\varepsilon}}\right).
\end{equation}
Then, by Theorem \ref{prop:crucari} and \eqref{eq:majexpN}, there exist constants $c_{K,\varepsilon}^{(1)},c_{K,\varepsilon}^{(2)},c_{K,\varepsilon}^{(3)}>0$ such that
\begin{align*}
\# (\mathfrak{p}+a)_B \leq D_{\varepsilon,B}&:=
c_{K,\varepsilon}^{(1)} B(\log_q(B))^{\#\cS_\infty-1-\frac{1}{1-\varepsilon}}
+c_{K,\varepsilon}^{(2)} (\log_q(B))^{\#\cS_\infty-\frac{1}{1-\varepsilon}}+c_{K,\varepsilon}^{(3)}B(\log_q(B))^{\#\cS_\infty-2-\frac{\varepsilon}{1-\varepsilon}}.
\end{align*}
 whenever ${\rm N}(\mathfrak{p}) \geq \theta_{\varepsilon}(B) \geq 2$. Moreover, $$
 D_{\varepsilon,B}=O_{K,\varepsilon}\left( B(\log_q(B))^{\#\cS_\infty-1-\frac{1}{1-\varepsilon}}\right),
 $$ and hence
\vskip 1mm
\hskip 20mm
$$ \frac{D_{\varepsilon,B}}{\#\mathcal{O}_{K,B}}=
O_{K,\varepsilon}(\log_q(B)^{-\frac{1}{1-\varepsilon}}).
$$ 
\noindent
Therefore,
\vskip 1mm

$$ 
\frac{D_{\varepsilon,B} \cdot \frac{\log_q(B)}{\log_q(\theta_{\varepsilon}(B))}}{\#\mathcal{O}_{K,B}}=
O_{K,\varepsilon}\left(\log_q(B)^{-\frac{\varepsilon}{1-\varepsilon}}/\log_q(\log_q(B))\right) \xrightarrow{B \rightarrow \infty} 0.
$$ 
\end{proof} 

\subsection{Auxiliary lemma} 

The following lemma will be needed.
\begin{lem}\label{lem:funfunc}
Let $\underline{T}=(T_1,\dots,T_k)$ with $k \geq 1$, and let $Q(\underline{T}) \in \mathcal{O}_K[\underline{T}]$ be a polynomial. Then there exists a constant $\gamma_{Q,\fd}$, depending only $Q$ and $\fd$, such that, for every $\underline{m} \in \mathcal{O}^k_{K,B}$ with $Q(\underline{m}) \neq 0$, the number $\nu_A$ of prime ideals $\fp$ satisfying 
$${\rm N}(\fp) \geq A\textrm{ and }\fp \mid Q(\underline{m})
$$ is at most 
$$
\frac{\gamma_{Q,\fd} \log_q(B)}{\log_q(A)}.
$$ 
\end{lem}

\begin{proof}
A straightforward estimate shows that there exists $\gamma_{Q,\fd}>0$ such that 
$$
|{\rm N}_{K/K_0}(Q(\underline{m}))| \leq  B^{\gamma_{Q,\fd}}.
$$ On the other hand,
$$
A^{\nu_A} \leq |{\rm N}_{K/K_0}(Q(\underline{m}))|.  
$$ Taking logarithms yields the desired bound.
\end{proof}

\section{Proof of Theorem \ref{prop:crucari}}\label{sec:crup} 
In this section, we prove Theorem \ref{prop:crucari}.

\subsection{Proof in the number field case} Assume that $K$ is a number field with signature $(r_1,r_2)$ and let $r=r_1+r_2-1=\#\cS_\infty-1$. Let $\sigma_1,\dots,\sigma_{r_1}: K \hookrightarrow \mathbb{R}$ and $\sigma_{r_1+1},\dots, \sigma_{r_1+r_2}: K \hookrightarrow \mathbb{C}$ denote the real and complex embeddings of $K$, respectively. For any element $\alpha \in \mathcal{O}_K$, define the embedding vector 
$$
\underline{\sigma}(\alpha)=(\sigma_1(\alpha),\dots,\sigma_{r_1+r_2}(\alpha)) \in \mathbb{R}^{r_1} \times \mathbb{C}^{r_2}.
$$ 
Define 
\begin{equation*}
Z(B)=\left\lbrace \underline{x}  \in \Rr^{r_1} \times \Cc^{r_2} \Bigg\vert \prod_{i=1}^{r_1+r_2} \max \{ 1 ,|x_i| \}^{d_i} \leq B\right\rbrace,
\end{equation*} where
$$
d_i=
\begin{cases}
1 & \textrm{if } 1 \leq i \leq r_1,\\
2 & \textrm{if } r_1+1 \leq i \leq r_1+r_2.
\end{cases}
$$
Observe that 
\begin{align*}
\#(\mathfrak{p}+a)_B&= \#\left( \underline{\sigma}(\mathfrak{p}+a) \cap Z(B)\right)\\
&=\#\left(\left( \underline{\sigma}(\mathfrak{p})+ \underline{\sigma}(a) \right) \cap Z(B) \right)\\
&=\#\left( \underline{\sigma}(\mathfrak{p}) \cap \left( Z(B)-\underline{\sigma}(a) \right) \right).
\end{align*}

Identifying $\mathbb{C}$ with $\mathbb{R}^2$, define the semi-algebraic set:
$$
Z=\{ (B,\underline{y},\underline{x}) \in \mathbb{R} \times (\mathbb{R}^{r_1} \times \mathbb{C}^{r_2})^2 \mid \underline{x} \in Z(B)- \underline{y}\}.
$$
We aim to apply \cite[Theorem 1.3]{BW14} to the lattice $\Lambda=\underline{\sigma}(\mathfrak{p})$ and the family of fibers
$$
Z_{S}=\{\underline{x} \in \mathbb{R}^{r_1} \times \mathbb{C}^{r_2} \mid (B,\underline{y},\underline{x}) \in Z\}=Z(B)-\underline{y},
$$ where $\underline{y}=\underline{\sigma}(a)$ and $S=(B,\underline{y})$. 
From that result, there exists $c_Z \in \mathbb{R}$ -- depending on $Z$, and hence on $r_1$ and $r_2$ -- such that 
\begin{equation}\label{eq:Ztz}
\left\vert \#(\mathfrak{p}+a)_B - \frac{{\rm Vol}(Z_S)}{\det \underline{\sigma}(\mathfrak{p})} \right\vert\leq\left\vert |Z_S \cap \underline{\sigma}(\mathfrak{p})|- \frac{{\rm Vol}(Z_S)}{\det \underline{\sigma}(\mathfrak{p})} \right\vert \leq c_Z \left(1+\sum_{j=1}^{n-1}\frac{V_j(Z_S)}{\lambda_1(\underline{\sigma}(\mathfrak{p}))}\right),
\end{equation}

\noindent
where ``${\rm Vol}$'' and ``${\rm det}$'' stand for volume and determinant respectively and  $V_1(Z_S), \ldots, V_{n-1}(Z_S)$ and $\lambda_1(\underline{\sigma}(\mathfrak{p}))$ are
described below.
\vskip 1mm

Note first that ${\rm Vol}(Z_{S})$ $={\rm Vol}(Z(B))$. Similarly, $V_j(Z_{S})=V_j(Z(B))$, where $V_j(Z_{S})$ denotes the sum of the $j$-dimensional volumes of the orthogonal projections of $Z_{S}$ on the various coordinate spaces obtained by equating any $n - j$ coordinates to zero, and $V_0(Z_{S}) = 1$, by convention; see \cite{BW14} for more details. Moreover, by \cite[Lemma 5.4.]{Bar14}, assuming $B \geq 2$, we have:
$$
V_j(Z_{S}) \leq O_\fd\left( B \log(B)^{r-1} \right)=O_\fd \left( \frac{{\rm Vol}(Z(B))}{\log(B)} \right).
$$

\noindent
The term $\lambda_1(\underline{\sigma}(\mathfrak{p}))$ is the first of the successive minima of the lattice $\underline{\sigma}(\mathfrak{p})$. The AM-GM inequality gives: 

\begin{align*}
\lambda_1(\underline{\sigma}(\mathfrak{p}))&:= \min \{ |\underline{v}| \mid \underline{v} \in \underline{\sigma}(\mathfrak{p}), \underline{v} \neq \underline{0}\}\geq {\rm N}(\mathfrak{p})^{1/\fd};
\end{align*}
(see, e.g., \cite[Lemma 11.1]{Wid16}). Furthermore, 
the determinant of the lattice $\underline{\sigma}(\mathfrak{p})$ is 
$$\det(\underline{\sigma}(\mathfrak{p}))=2^{-r_2}|\Delta_K|^{1/2}{\rm N}(\mathfrak{p}).
$$ 

 \noindent
 Since the number of possible $Z$ (determined by $r_1,r_2$ with $r_1+2r_2=\fd$) is finite, it follows that
$$
\#(\mathfrak{p}+a)_B = \frac{2^{r_2} {\rm Vol(Z(B))}}{|\Delta_K|^{1/2}{\rm N}(\mathfrak{p})}+O_\fd \left( 1+ \frac{{\rm Vol}(Z(B))}{{\rm N}(\mathfrak{p})^{1/\fd} \log(B)} \right).
$$
To conclude (1) of Theorem \ref{lem:bounheigharithprog}, a classical integration computation proves that
$$
{\rm Vol}(Z(B))=\frac{2^{r_1} \pi^{r_2} }{r!} B \log(B)^r+O_\fd(B \log(B)^{r-1}).
$$

The special case that $a=0$ and $\mathfrak{p} = \mathcal{O}_K$ yields a similar formula for $\# \mathcal{O}_{K,B}$. The second assertion of Theorem \ref{lem:bounheigharithprog} follows then immediately by Definition \ref  {defn:density}.

\subsection{Proof in the function field case} 

\subsubsection{Setup of the proof}
Assume that $K/\F_q(t)$ is a function field such that $K \cap \overline{\F_q}=\F_q$. The \emph{degree} $\deg(v)$ of $v$ is defined as the degree of the residue field of $K$ at $v$ over $\F_q$. 

As usual, $\ord_v(0)=\infty > m$ for all $m \in \Z$. Each $\underline{x} \in K^n$ defines a divisor
$$
(\underline{x})=\div(\underline{x})=\sum_{v \in \mathcal{M}_K} \ord_v(\underline{x}) \cdot v.
$$
Denote by ${\rm Div}(K)$ the set of all divisors of $K$. For a divisor $A=\sum_{v \in \cM_K } A_v \cdot v$, its \emph{support} is 
$$
{\rm supp}(A)=\{ v \in \cM_K \mid A_v \neq 0\},$$ 
and  write 
$$A_0=\sum_{v \in \cM_K, A_v>0} A_v \cdot v \quad \textrm{and} \quad A_\infty=\sum_{v \in \cM_K, A_v<0} -A_v \cdot v
$$ for its \emph{zero divisor} and \emph{pole divisor}, respectively, so that 
$$
A=A_0-A_\infty.
$$

Let $\cS$ be a nonempty finite set of places of $K$ and let 
$$
\rho_{\cS}=\gcd\{\deg(v)\mid v\in\cS\}.
$$
\begin{defn}\label{defn:ringofint}
The \emph{ring of $\cS$-integers} of $K$ is the set of elements of $K$ that have no poles outside of $\cS$, i.e., 
$$
\mathcal{O}_{\cS}=\{\alpha \in K \mid \ord_v(\alpha) \geq 0\,\,\textrm{for all } v \notin \cS \}=\{\alpha \in K \mid {\rm supp}((\alpha)_\infty) \subset \cS\}.
$$ 
\end{defn}

\begin{rmk} For $\alpha \in K^\times$, we have 
\begin{equation}\label{eq:hkdeg}
h_K(\alpha)=\deg((\alpha)_\infty)=-\deg(\div(1,\alpha))=\sum_{v \in \mathcal{M}_K}\max\{0,-\ord_v(\alpha)\} \deg(v).
\end{equation}
\end{rmk}
As mentioned earlier, we obtain a stronger version of Theorem \ref{prop:crucari} by counting points of exact height. In fact, since $h_K(\cO_\cS) \subset \N \rho_{\cS}$, the possible height values are naturally indexed by the nonnegative integers, and it is more convenient to count the number of points of logarithmic height exactly $\mathfrak{b}\rho_{\cS}$. To this end, for $\mathfrak{A} \subset \cO_K$ and $\fb$, introduce the following counting function:

$$
N^{=}_{h_K}(\mathfrak{A},\mathfrak{b}\rho_{\cS})
=
\#\{\alpha\in\cO_K\mid h_K(\alpha)=\mathfrak{b}\rho_{\cS}\}.
$$
\begin{thm}\label{cor:funcfiel} 
Let $\fp$ be an ideal of $\cO_\cS$, let $a \in \cO_\cS$ and set $Q=\sum_{v \in \cS} v \in {\rm Div}(K)$. Then for $\mathfrak{b}\in \N$,
$$
N^{=}_{h_K}(\mathfrak{p}+a,\mathfrak{b}\rho_\cS) =  \frac{\tilde \kappa_{K,\cS}}{{\rm N}(\fp)} \cdot q^{\mathfrak{b}\rho_\cS} \cdot (\mathfrak{b}\rho_\cS)^{\#\cS-1}+ O_{K,\cS}( \tilde E_{\cS}(\mathfrak{b},\mathfrak{p})), 
$$ where
$$
\tilde \kappa_{K,\cS}=\frac{\rho_\cS}{(\#\cS-1)! q^{g_K-1} \prod_{v \in \cS} \deg(v)} \cdot \prod_{v \in \cS} \left( 1 - \frac{1}{q^{\deg(v)}}\right),
$$ and
$$
\tilde E_{\cS}(\mathfrak{b},\fp)=
\begin{cases}
\mathfrak{b}^{\#\cS-1}+ \frac{q^{\mathfrak{b}\rho_\cS} \deg(\fp)}{{\rm N}(\fp)^{1/\deg(Q)}}\mathfrak{b}^{\#\cS-2}& \textrm{if } \#\cS >1,\\
1+\frac{1}{{\rm N}(\fp)^{1/2}}\cdot q^{\frac{\mathfrak{b}\rho_\cS}{2}}&\textrm{if } \#\cS=1.
\end{cases}
$$
\end{thm}

Theorem \ref{prop:crucari} follows then from the following by noting that $\deg(Q) \leq  \fd$ when $\cS=\cS_\infty$.
\begin{cor}
With the same notation as in Theorem \ref{cor:funcfiel}, for $B \in q^{\rho_\cS\N}$,
$$
\#(\mathfrak{p}+a)_B
=
\frac{\kappa_{K,\cS}}{{\rm N}(\fp)}
\log_q(B)^{\#\cS-1}B
+
O_{K,\cS}\left(
E_\cS(B,\fp)
\right),
$$
where $\kappa_{K,\cS}=\frac{q^{\rho_\cS}}{q^{\rho_\cS}-1} \tilde \kappa_{K,\cS}$ and
$$
E_\cS(B,\fp)=
\begin{cases}
\log_q(B)^{\#\cS}
+
\frac{\deg(\fp)}
{{\rm N}(\fp)^{1/\deg(Q)}}
B\log_q(B)^{\#\cS-2} &\textrm{if }\#\cS>1,\\
\log_q(B)
+
\frac{1}{{\rm N}(\fp)^{1/2}}B^{1/2}&\textrm{if }\#\cS=1.
\end{cases}
$$
\end{cor}

\begin{proof}
Note that
$$
\#(\mathfrak{p}+a)_B
=
\sum_{\mathfrak{b}=0}^{\log_q (B)/\rho_\cS}
N^{=}_{h_K}(\mathfrak{p}+a,\mathfrak{b}\rho_\cS).
$$
The result follows by applying summation by parts, together with the estimate for $N^{=}_{h_K}(\mathfrak{p}+a,\mathfrak{b}\rho_\cS)$ given in Theorem \ref{cor:funcfiel}. Indeed, for $\#\cS \geq 2$, we have
\begin{align*}
\#(\fp+a)_B&=\frac{\tilde \kappa_{K,\cS}}{{\rm N}(\fp)} \rho_\cS^{\#\cS-1}\left(\frac{\log_q(B)}{\rho_\cS}\right)^{\#\cS-1}B\left(\frac{q^{\rho_\cS}}{q^{\rho_\cS}-1}+O_\cS(\log_q(B)^{-1}) \right)\\
&+ O_\cS \left( \log_q(B)^{\#\cS}+\frac{\deg(\fp)}{{\rm N}(\fp)^{1/\deg(Q)}}B\log_q(B)^{\#\cS-2}\right)\\
&=\frac{\kappa_{K,\cS}}{{\rm N}(\fp)}
\log_q(B)^{\#\cS-1}B
+
O_{K,\cS}\left(
E_\cS(B,\fp)
\right).
\end{align*}
For $\#\cS=1$, we have
\begin{align*}
\#(\fp+a)_B&=\frac{\tilde \kappa_{K,\cS}}{{\rm N}(\fp)}\frac{q^{\rho_\cS}B-1}{q^{\rho_\cS}-1}+O_\cS\left(\log_q(B)^{\#\cS}+\frac{1}{{\rm N}(\fp)^{1/2}}B^{1/2}\right)\\
&=\frac{\kappa_{K,\cS}}{{\rm N}(\fp)}
\log_q(B)^{\#\cS-1}B
+
O_{K,\cS}\left(
E_\cS(B,\fp)
\right).
\end{align*}
\end{proof}

The remaining of this section is then  devoted to the proof of Theorem \ref{cor:funcfiel}.

\subsubsection{Preliminaries}

\subsubsection*{$**$ Riemann--Roch}

For each divisor $A \in {\rm Div}(K)$, the \emph{Riemann--Roch space associated with $A$} is the $\F_q$-vector space
$$
\mathscr{L}(A)=\{ \alpha \in K \mid (\alpha)+A \geq 0 \}=\{\alpha \in K \mid \ord_v(\alpha) \geq -A_v\,\,\textrm{for all } v \in \cM_K\},
$$ and we set $\ell(A):=\dim_{\F_q}(\mathscr{L}(A))$. Recall the following well-known results:
\vskip 1mm
\noindent
$\bullet$ If $\deg(A)<0$, then $\ell(A)=0$.
\vskip 1mm
\noindent
$\bullet$ ({\bf Clifford's theorem}) If $0 \leq \deg(A) < 2 g_K-1$, then
$$
\ell(A) \leq 1 + \frac{1}{2} \deg(A).
$$
\vskip 1mm
\noindent
$\bullet$ ({\bf Riemann--Roch theorem}) If $\deg(A) \geq 2g_K-1$, then
$$
\ell(A)=\deg(A)-g_K+1. 
$$

\subsubsection*{$**$ $\mathcal{S}$-integers}
A nonzero ideal $\mathfrak{p}$ of $\mathcal{O}_{\cS}$ corresponds to an effective divisor $\sP \in {\rm Div}(K)$ (that is, $\sP_v \geq 0$ for all $v \in \cM_K$) with ${\rm supp}(\sP) \cap \cS=\emptyset$ such that
\begin{equation}\label{eq:defP}
\mathfrak{p}=\{\alpha \in K \mid \ord_v(\alpha) \geq \sP_v\,\,\textrm{for all }v \notin \cS \}.
\end{equation}
More precisely,
$$
\sP_v=
\begin{cases}
\min\{ \ord_v(\alpha) \mid \alpha \in \mathfrak{p}\} &\textrm{if } v \notin \cS,\\
0 &\textrm{otherwise}.
\end{cases}
$$
Note that $\deg(\sP)=\deg(\mathfrak{p})=\log_q({\rm N}(\fp))$. In the following, we identify $\sP$ with $\mathfrak{p}$ by abuse of notation.

\subsubsection{Points of bounded heights in arithmetic progression}

Fix a nonempty finite set $\cS \subset \cM_K$. Let $\mathfrak{p}$ be an ideal of $\cO_{\cS}$ and let $a \in \cO_{\cS}$. Note that $h_K(\cO_{\cS}) \subset \rho_{\cS} \mathcal{U}$, where
$$
\mathcal{U}=\sum_{v \in \cS} \N \cdot \frac{\deg(v)}{\rho_\cS}.
$$
By Schur's theorem, all sufficiently large integers belongs to $\mathcal{U}$. For $\fb \in \N$, define
$$
\mathbb{B}_\fb:=\left\lbrace \underline{b} \in \N_0^{\#\cS} \bigg\vert \sum_{v \in \cS} b_v \frac{\deg(v)}{\rho_\cS}=\fb  \right\rbrace.
$$
Moreover, for $\underline{b} \in \N_0^{\#\cS}$, define
$$
(\mathfrak{p}+a)_{=\underline{b}}=\left\lbrace \theta \in \mathfrak{p}+a \bigg\vert (\theta)_{\infty}=\sum_{v \in \cS} b_v \cdot v\right\rbrace.
$$
\begin{rmk}
For
$\gamma \in \mathfrak{p}+a$ with $h_K(\gamma)=\fb \rho_\cS$, there exists a unique 
$\underline{b}=(b_v)_{v \in \cS} \in \mathbb{B}_\fb$ such that $\gamma \in (\mathfrak{p}+a)_{=\underline{b}}$, namely
$$
b_v=\begin{cases}
-\ord_v(\gamma)& \textrm{if } v \in \cS \textrm{ and } \ord_v(\gamma)<0,\\
0 & \textrm{if } v \in \cS \textrm{ and } \ord_v(\gamma)\geq 0.
\end{cases}
$$
\end{rmk}
By the above remark, we obtain the disjoint union
\begin{equation}\label{eq:22}
(\mathfrak{p}+a)_{=\fb}:=\bigsqcup_{\underline{b} \in \mathbb{B}_\fb}(\mathfrak{p}+a)_{=\underline{b}},
\end{equation}
and hence
\begin{equation}\label{eq:33}
N^{=}_{h_K}(\mathfrak{p}+a,\mathfrak{b}\rho_\cS)=\#(\mathfrak{p}+a)_{=\fb}=\sum_{\underline{b} \in \mathbb{B}_\fb} \#(\mathfrak{p}+a)_{=\underline{b}}.
\end{equation}

\subsubsection*{$**$ Asymptotics of $\mathbb{B}_\fb$} 

Recall that 
$$
\fb \in \mathcal{U}=\sum_{v \in \cS} \N_0 \cdot \frac{\deg(v)}{\rho_\cS};
$$ For $\underline{b} \in \mathbb{B}_\fb$, set 
$$
\cS_{\underline{b}}:=\{ v \in \cS \mid b_v>0\}.
$$
We split $\mathbb{B}_{\fb}$ into two sets:

\begin{align}\label{eq:defb00}
\mathbb{B}_{\fb}^{(0)}&=\{ \underline{b} \in \mathbb{B}_\fb  \mid \cS_{\underline{b}}=\cS\textrm{ and } b_v> - \ord_v(a)\,\,\textrm{for all } v \in \cS\}\\
&=\{ \underline{b} \in \mathbb{B}_\fb \mid b_v > \max \{0,-\ord_v(a) \}\,\, \textrm{for all } v \in \cS \},\nonumber
\end{align}
and $\mathbb{B}_\fb^{(1)}=\mathbb{B}_\fb \setminus \mathbb{B}_\fb^{(0)}$. 
Write $\mathbb{B}_\fb^{(1)}=\mathbb{B}_\fb^{(11)} \cup \mathbb{B}_\fb^{(12)}$, where
$$
\mathbb{B}_\fb^{(11)}=\{\underline{b} \in \mathbb{B}_\fb \mid b_v \leq \max\{0,-\ord_v(a)\}\,\,\textrm{for all }v \in \mathcal{S}\},
$$ and
$$
\mathbb{B}_\fb^{(12)}=\{\underline{b} \in \mathbb{B}_\fb \mid b_v \leq \max\{0,-\ord_v(a)\}\,\,\textrm{and}\,\, b_w > \max\{0,-\ord_w(a)\}\,\,\textrm{for some }v,w \in \mathcal{S}\}.
$$
This splitting allows us to isolate the main term and also to exhibit the exponent of ``${\rm N}(\fp)$'' appearing in the error term of Theorem \ref{prop:crucari}, which is crucial in the proof of the Ekedahl--Poonen formula. From \eqref{eq:33}, it follows that

\begin{equation}\label{eq:44}
N^{=}_{h_K}(\mathfrak{p}+a,\mathfrak{b}\rho_\cS)=\sum_{\underline{b} \in \mathbb{B}_\fb^{(0)}}\#(\mathfrak{p}+a)_{=\underline{b}}+ \sum_{\underline{b} \in \mathbb{B}_\fb^{(1)}} \#(\mathfrak{p}+a)_{=\underline{b}}.
\end{equation}
Our goal is to analyze $\# \mathbb{B}_\fb^{(0)}$ and $\# \mathbb{B}_\fb^{(1)}$. Let $s=\#\cS$ and write 
$$\cS=\{ v_1,\dots,v_s\}.$$ Define $\underline{\xi}=\left( \frac{\deg(v_1)}{\rho_\cS},\dots,\frac{\deg(v_s)}{\rho_\cS}\right) \in \N^s$, which is {primitive}. Let $\langle \cdot, \cdot \rangle$ denote the standard inner product on $\mathbb{R}^s$. Write
$$
\fb=\sum_{i=1}^s b_i \cdot \xi_i =\langle \underline{b},\underline{\xi}\rangle.
$$ Consider the hyperplane
$$
H_\fb= \{ \underline{x} \in \mathbb{R}^s \mid \langle \underline{x},\underline{\xi}\rangle=\fb\},
$$ and the simplex in $H_\fb$ defined by $\sS_\fb=H_\fb \cap \mathbb{R}^s_{\geq 0}$. Since $\underline{\xi} \in \Z^s$ is primitive, Bézout's identity implies that $\Z^s \cap H_\fb$ is nonempty; fix $\underline{z}$ in this intersection. Define the translation 
$$
T(\underline{x})=\underline{x}-\underline{z}.
$$ Then $T(\Z^s)=\Z^s$, and therefore
$$
\#(\sS_\fb \cap \Z^s)=\#(T(\sS_\fb \cap \Z^s))=\#(T(\sS_\fb) \cap \Z^s).
$$
Let $G_0:=T(H_\fb)=(\mathbb{R}\cdot \underline{\xi})^{\perp}$, and define 
$$
\Lambda_{\underline{\xi}}=G_0 \cap \Z^s,$$ which is a lattice of full rank in the $(s-1)$-dimensional subspace $G_0$. Since $\underline{\xi}$ is primitive, we have
\begin{equation}\label{eq:formdet}
\det(\Lambda_{\underline{\xi}})=\Vert \underline{\xi} \Vert,
\end{equation} where $\Vert \cdot \Vert$ denotes the Euclidean norm.

\begin{lem}\label{lem:volsm}
The volume of $\sS_\fb$ is
$$
{\rm Vol}_{s-1}(\sS_\fb)=\frac{\Vert \underline{\xi} \Vert \fb^{s-1}}{(s-1)! \xi_1 \cdots \xi_s}.
$$
\end{lem}

\begin{proof}
Define 
$$
f(x_1,\dots,x_{s-1})=\frac{\fb}{\xi_s}-\frac{\xi_1}{\xi_s}x_1-\dots-\frac{\xi_{s-1}}{\xi_s}x_{s-1}.
$$
The volume of $\sS_\fb$ is obtained by integrating the function
$$
\sqrt{1+\left(\frac{\partial f}{\partial x_1}\right)^2+\dots+\left(\frac{\partial f}{\partial x_{s-1}}\right)^2}=\frac{\Vert \underline{\xi} \Vert}{\xi_s}
$$
over a certain coordinate domain. More precisely,
$$
{\rm Vol}_{s-1}(\sS_\fb)=\frac{\Vert \underline{\xi} \Vert}{\xi_s} 
\int_{0}^{\frac{\fb}{\xi_1}} \int_{0}^{\frac{\fb}{\xi_2}-\frac{\xi_1}{\xi_2}x_1}\cdots \int_{0}^{\frac{\fb}{\xi_{s-1}}-\frac{\xi_1}{\xi_{s-1}}x_1 -\cdots-\frac{\xi_{s-2}}{\xi_{s-1}}x_{s-2}}d{x_{s-2}}\cdots dx_1.
$$
Denote this last integral by $I_\fb(\xi_1,\dots,\xi_s)$ so that
$$
{\rm Vol}_{s-1}(\sS_\fb)=\frac{\Vert \underline{\xi} \Vert}{\xi_s} I_\fb(\xi_1,\dots,\xi_s).
$$
We need to prove that
$$
I_\fb(\xi_1,\dots,\xi_s)=\frac{\fb^{s-1}}{(s-1)! \xi_1 \cdots \xi_{s-1}}.
$$
This follows  by induction on $s$, using the recursion
$$
I_\fb(\xi_1,\dots,\xi_s)=\int_0^{\frac{\fb}{\xi_1}}I_{\fb-\xi_1 x_1}(\xi_2,\dots,\xi_s)\, dx_1,
$$
with base case ($s=2$) $I_\fb(\xi_1,\xi_2)=\int_0^{\frac{\fb}{\xi_1}}dx_1=\frac{\fb}{\xi_1}$.
\end{proof}
The $(s-2)$-dimensional boundary of $\sS_\fb$ consists of $s$ faces $F_1,\dots,F_s$, each lying in a translate of a parallelepiped
$$
[0,1[\cdot u_1 + \dots+ [0,1[ \cdot u_{s-2}
$$
with $\Vert u_i\Vert \leq \fb$. Each such a parallelepiped can be covered by the image of a Lipschitz map
$$
\phi:[0,1]^{s-2} \rightarrow G_0
$$
with Lipschitz constant 
\begin{equation}\label{eq:lipcons}
\sqrt{\sum_{i=0}^{s-2} \fb^2} \leq \sqrt{s-2} \fb.
\end{equation} This property obviously remains true for the translated faces $F_1,\dots,F_s$.

\begin{lem}\label{lem:dedea}
We have
$$
\# \mathbb{B}_\fb=\frac{\fb^{s-1} \rho_\cS^s}{(s-1)! \prod_{v \in \cS}\deg(v)}+ O_{\cS}\left( 1 + \fb^{s-2} \right).
$$
\end{lem}

\begin{proof}
We apply \cite[Lemma 2]{MV07}. To this end, we note that $\Lambda_{\underline{\xi}} \subset \Z^s$, hence every non-zero vector has length at least one. Using \eqref{eq:formdet}, \eqref{eq:lipcons} and Lemma \ref{lem:volsm}, we obtain 
\begin{align*}
\#\mathbb{B}_\fb&=\#(\Lambda_{\underline{\xi}} \cap T(\sS_\fb))=\frac{{\rm Vol}_{s-1}(T(\sS_\fb))}{\det(\Lambda_{\underline{\xi}})}+O_{\mathcal{S}}(1+\fb^{s-2})\\
&=\frac{\fb^{s-1}}{(s-1)! \xi_1 \cdots \xi_s}+O_{\mathcal{S}}(1+\fb^{s-2}).
\end{align*}
\end{proof}

\begin{cor}\label{cor:dedea}
We have
$$
\# \mathbb{B}_\fb^{(1)}=O_{\mathcal{S}}((1+h_K(a))(1+\fb^{s-2})).
$$
\end{cor}

\begin{proof}
Since $h_K(a) \geq \max \{0,-\ord_{v}(a)\}$ for all $v \in \cS$, we have
\begin{equation}\label{eq:boundB1}
\#\mathbb{B}_\fb^{(1)} \leq \sum_{i=1}^s \# \{ \underline{b} \in \mathbb{B}_\fb \mid b_i \leq h_K(a)\}.
\end{equation}
Fix $i=1,\dots,s$. For $m \geq 0$, let $\mathbb{B}^{(i)}_{\fb,m}=\{ \underline{b} \in \mathbb{B}_\fb \mid b_i=m \}$. Then
$$
\# \mathbb{B}^{(i)}_{\fb,m}=\#\left \lbrace \underline{\tilde b} \in \N_0^{s-1} \bigg \vert \sum_{\ell=1, \ell \neq i}^{s} \tilde{b}_\ell \cdot \deg(v_\ell)=\fb \rho_\cS- m \deg(v_i) \right \rbrace.
$$
By Lemma \ref{lem:dedea} applied to $\cS \setminus \{v_i\}$, we obtain
$$
\mathbb{B}^{(i)}_{\fb,m} \leq \fb^{s-2}+ O_{\mathcal{S}}(1+\fb^{s-3}).
$$
Summing over $m \leq h_K(a)$ yields
$$
\#\{\underline{b} \in \mathbb{B}_\fb \mid b_i \leq h_K(a)\}=\sum_{m=0}^{h_K(a)} \# \mathbb{B}^{(i)}_{\fb,m} \leq \sum_{m=0}^{h_K(a)} \left(\fb^{s-2}+O_{\cS}(1+\fb^{s-3})\right)= O_{\cS} ((h_K(a)+1)(1+\fb^{s-2})).
$$
Summing over $i$, from \eqref{eq:boundB1}, we get 
$$
\# \mathbb{B}_\fb^{(1)} \leq \sum_{i=1}^s O_{\cS} ((h_K(a)+1)(1+\fb^{s-2}))=O_{\cS}((1+h_K(a))(1+\fb^{s-2})).
$$
\end{proof}

\begin{cor}\label{cor:444a}
We have
$$
\#\mathbb{B}_\fb^{(0)}=\frac{\fb^{s-1}\rho_\cS^{s}}{(s-1)! \prod_{v \in \cS} \deg(v)}+ O_{\cS} ((1+h_K(a))(1+\fb^{s-2})).
$$
\end{cor}

\begin{proof}
Since $\mathbb{B}_\fb^{(1)}=\mathbb{B}_\fb \setminus \mathbb{B}_\fb^{(0)}$, this follows immediately from Lemma \ref{lem:dedea} and Corollary \ref{cor:dedea}.
\end{proof}

\subsubsection*{$**$ Asymptotics of  $(\mathfrak{p}+a)_{=\underline{b}}$} Let $\sP \in \cM_K$ be as in \eqref{eq:defP}. Let $\underline{b} \in \mathbb{B}_\fb$ and recall that
$$
\cS_{\underline{b}}=\{ v \in \cS \mid b_v >0\}.
$$

\begin{lem}\label{lem:Bb} Let $\underline{b} \in \mathbb{B}_\fb^{(12)}$. Then $$
\#(\fp+a)_{=\underline{b}} \leq 1+q^{\fb \rho_\cS-\deg(\fp)+h_K(a)-R(a)+1},
$$ where
\begin{equation}\label{eq:defRa}
R(a):=\min_{w \in \mathcal{S}} (\max\{0,-\ord_w(a)\}\deg(w)).
\end{equation}
\end{lem}

\begin{proof}
Let $w \in \mathcal{S}$ be such that $b_w > \max\{0,-\ord_w(a)\}$. Define a divisor $Q \in {\rm Div}(K)$ by 
\begin{equation}\label{eq:defqw}
Q=
\begin{cases}
b_w & \textrm{if } v=w,\\
\max\{b_v,-\ord_v(a)\} &\textrm{if } v \in \cS_{\underline{b}}\setminus\{w\},\\
\max\{0,-\ord_v(a)\} &\textrm{if }v \in \cS \setminus (\cS_{\underline{b}}\cup\{w\}),\\
-\sP_v &\textrm{if }v\in\cM_K\setminus \cS.
\end{cases}
\end{equation} 

We first claim that $(\mathfrak{p}+a)_{=\underline{b}}-a \subset \mathscr{L}(Q)$, hence $(\fp +a)_{=\underline{b}} \subset \mathscr{L}(Q)+a$. Indeed, let $\alpha \in (\mathfrak{p}+a)_{=\underline{b}}-a \subset \mathfrak{p}$. By definition we have:
\vskip 1mm
\noindent
$\bullet$ $\ord_v(\alpha+a)=-b_v$ for all $v \in \cS_{\underline{b}}$;
\vskip 1mm
\noindent
$\bullet$ $\ord_v(\alpha+a) \geq 0$ for all $v \in \cS \setminus \cS_{\underline{b}}$.
\vskip 1mm
\noindent
Since 
$$
\ord_v(\alpha)=\ord_v((\alpha+a)-a) \geq \min\{\ord_v(\alpha+a),\ord_v(a)\},
$$ for all $v \in \cM_K$, it follows that:
\vskip 1mm
\noindent
$\bullet$ $\ord_w(\alpha)=-b_w$;
\vskip 1mm
\noindent
$\bullet$ $\ord_v(\alpha) \geq \min\{-b_v,\ord_v(a)\}$ for all $v \in \cS_{\underline{b}} \setminus \{w\}$;
\vskip 1mm
\noindent
$\bullet$ $\ord_v(\alpha) \geq \min\{0,\ord_v(a)\}$ for all $v \in \cS \setminus(\cS_{\underline{b}} \cup \{w\})$.
\vskip 1mm
\noindent
Moreover, since $\alpha \in \mathfrak{p}$, by \eqref{eq:defP} we have $\ord_v(\alpha) \geq \sP_v$ for all $v \in \cM_K \setminus \cS$. This proves the claim.

To complete the proof of the lemma, observe that
$$
Q \leq -\mathscr{P}+\sum_{v \in \cM_K} b_v \cdot v +(a)_{\infty}-\max\{0,-\ord_w(a)\}\cdot w.
$$ Hence, by \eqref{eq:hkdeg}, $\deg(Q) \leq \fb \rho_\cS-\deg(\fp)+h_K(a)-\max\{0,-\ord_w(a)\}\deg(w)$.
By the Riemann--Roch theorem together with Clifford's theorem, it follows that
$$
\ell(Q) \leq 
\begin{cases}
\fb \rho_\cS-\deg(\fp)+h_K(a)-\max\{0,-\ord_w(a)\}\deg(w)+1 &\textrm{if }\deg(Q) \geq 0,\\
0 & \textrm{ otherwise}.
\end{cases}
$$
Consequently,
$$
\#(\fp+a)_{=\underline{b}} \leq \#(\mathscr{L}(Q))  \leq  1+q^{\fb \rho_\cS-\deg(\fp)+h_K(a)-\max\{0,-\ord_w(a)\}\deg(w)+1} \leq 1+q^{\fb \rho_\cS-\deg(\fp)+h_K(a)-R(a)+1}.
$$ This completes the proof.

\end{proof}

\begin{lem}\label{lem:Bb11}
Define a divisor $Q \in {\rm Div}(K)$ by 
\begin{equation}\label{eq:defQ}
Q_v=
\begin{cases}
\max\{0,-\ord_v(a)\} &\textrm{if }v \in \cS,\\
-\sP_v &\textrm{if }v\in\cM_K\setminus \cS.
\end{cases}
\end{equation}
Then 
$$
\bigcup_{\underline{b} \in \mathbb{B}^{(11)}}(\mathfrak{p}+a)_{=\underline{b}} \subset \mathscr{L}(Q)+a.
$$ Moreover
$$
\#\mathscr{L}(Q)\leq 1+q^{-\deg(\fp)+h_K(a)+1}.
$$
\end{lem}

\begin{proof}
Let $\underline{b} \in \mathbb{B}_\fb^{(11)}$. By an argument similar to that used in Lemma \ref{lem:Bb}, we obtain
$(\mathfrak{p}+a)_{=\underline{b}}-a \subset \mathscr{L}(Q)$, hence $(\fp+a)_{=\underline{b}} \subset \mathscr{L}(Q)+a$. This proves the first statement.

For the second statement, observe that
$$
Q \leq -\mathscr{P} +(a)_{\infty}.
$$ Therefore $\deg(Q) \leq -\deg(\fp)+h_K(a)$ by \eqref{eq:hkdeg}.
By the Riemann--Roch theorem together with Clifford's theorem, it follows that
$$
\ell(Q) \leq 
\begin{cases}
-\deg(\fp)+h_K(a)+1 &\textrm{if }\deg(Q) \geq 0,\\
0 & \textrm{ otherwise}.
\end{cases}
$$
Consequently,
$$
\#\mathscr{L}(Q)  \leq  1+q^{-\deg(\fp)+h_K(a)+1},
$$ as claimed.
\end{proof}

\begin{lem}\label{cor:22as}
Suppose that $\underline{b} \in \mathbb{B}_\fb^{(0)}$. Then
$$
\#(\mathfrak{p}+a)_{=\underline{b}}=\left( \prod_{v \in \cS} \left(1 -\frac{1}{q^{\deg(v)}}\right) \right) \cdot \frac{1}{q^{\deg(\fp)}} \cdot q^{\fb \rho_\cS-g_K+1}+ O_{\cS} \left( 1 + q^{\frac{\fb \rho_\cS-\deg(\fp)}{2}} \right).
$$
\end{lem}

\begin{proof} We first claim that
$$
(\mathfrak{p}+a)_{=\underline{b}} - a
=
\left\{
\alpha \in K \;\middle|\;
\begin{array}{l}
\ord_v(\alpha) = -b_v \quad \text{for all } v \in \cS_{\underline{b}},\\[2pt]
\ord_v(\alpha) \ge \sP_v \quad \text{for all } v \in \cM_K \setminus \cS
\end{array}
\right\}=(\fp)_{=\underline{b}}.
$$ Indeed, $$
(\mathfrak{p}+a)_{=\underline{b}}
=
\Bigl\{
\alpha + a \in K \;\Big|\;
\begin{array}{l}
\ord_v(\alpha + a) = -b_v \quad \text{for all } v \in \cS,\\[1mm]
\ord_v(\alpha) \ge \sP_v \quad \text{for all } v \in \cM_K \setminus \cS
\end{array}
\Bigr\}.
$$
Fix $v \in \cS$. Since $\ord_v(a) > -b_v$ by \eqref{eq:defb00}, we have
$$
-b_v = \ord_v(\alpha+a) \ge \min\{\ord_v(\alpha), \ord_v(a)\},
$$
and hence $\ord_v(\alpha) = -b_v$. Conversely, $\ord_v(\alpha) = -b_v$ clearly implies
$\ord_v(\alpha+a) = -b_v$. This completes the proof of the claim.

To complete the proof of the lemma, let $Q_{\underline{b}}=\sum_{v \in \cS} b_v \cdot v$. For $I \subset \cS$, let $e_I=\sum_{v \in I} v \in \cM_K$. Then
$$
\mathfrak{p}_{=\underline{b}}=\mathscr{L}(Q_{\underline{b}}-\sP) \setminus \left( \bigcup_{v \in \cS } \mathscr{L}(Q_{\underline{b}}-e_{\{v\}}-\sP)\right).
$$
By the inclusion-exclusion principle, it follows that
$$
\# \mathfrak{p}_{=\underline{b}}=\sum_{I \subset \cS} (-1)^{\# I} \# \mathscr{L}(Q_{\underline{b}}-e_{I}-\sP). 
$$  

Let $I \subset \cS$. Then
\begin{equation}\label{eq:Bund}
\deg(Q_{\underline{b}}-e_{I}-\sP)=\fb \rho_\cS- \sum_{v \in I} \deg(v) - \deg(\fp)
\end{equation}
\vskip 1mm
\noindent
$\bullet$ If $\eqref{eq:Bund} < 2 g_K-1$, then by Clifford's Theorem,
$$
\#\mathscr{L}(Q_{\underline{b}}-e_I-\sP) \leq 1+ q^{\frac{1}{2}\left(\fb \rho_\cS-\sum_{v \in I}\deg(v) -\deg(\fp)\right)+1} \leq 1 + q^{\frac{\fb \rho_\cS-\deg(\fp)}{2}+1}.
$$ Moreover $$
q^{\fb \rho_\cS-\sum_{v \in I} \deg(v) -\deg(\fp)-g_K+1} \leq q^{g_K}.
$$
\vskip 1mm
\noindent
$\bullet$ If $\eqref{eq:Bund} \geq 2 g_K-1$, then by the Riemann--Roch theorem,
$$
\#\mathscr{L}(Q_{\underline{b}}-e_I-\sP)=q^{\fb \rho_\cS-\sum_{v \in I} \deg(v)-\deg(\fp)-g_K+1}.
$$
Combining these estimates gives
\begin{align*}
\#\mathfrak{p}_{=\underline{b}}&=\sum_{I \subset \cS} (-1)^{\#I} \# \mathscr{L}(Q_{\underline{b}}-e_I-\sP)\\
&=\sum_{I \subset \cS} \left( (-1)^{\#I} \left[ q^{-\sum_{v \in I} \deg(v)} \cdot q^{\fb \rho_\cS-g_K+1-\deg(\fp)} \right] + O\left( 1 + q^{g_K}+ q^{\frac{\fb \rho_\cS-\deg(\fp)}{2}} \right) \right)\\
&=\left( \prod_{v \in \cS} \left( 1 - \frac{1}{q^{\deg(v)}} \right) \right) \cdot q^{\fb \rho_\cS-g_K+1-\deg(\fp)}+ O_{\cS} \left(1+q^{\frac{\fb \rho_\cS-\deg(\fp)}{2}}\right).
\end{align*}
Finally, since $\underline{b} \in \mathbb{B}_\fb^{(0)}$, the claim above gives
$$
\#(\mathfrak{p}+a)_{=\underline{b}}=\#(\mathfrak{p})_{=\underline{b}}, 
$$ which completes the proof. 
\end{proof}

\subsubsection*{$**$ Asymptotics of $(\mathfrak{p}+a)_{=\fb}$}
Recall that $\fb \in \mathcal{U}=\sum_{v \in \cS} \N_0 \cdot \frac{\deg(v)}{\rho_\cS}$ and that $\rho_\cS=\gcd\{\deg(v) \mid v \in \cS\}$.

\begin{prop}\label{prop:termcruc}
We have $$
N^{=}_{h_K}(\mathfrak{p}+a,\mathfrak{b}\rho_\cS)= \tilde \kappa_{K,\cS} \cdot B^{s-1}\cdot q^{\fb \rho_\cS-\deg(\fp)-g_K+1}+ O_{\cS}(E_\cS(\fb,a,\mathfrak{p})), 
$$ where
$$
\tilde \kappa_{K,\cS}=\frac{\rho_\cS}{(s-1)! \prod_{v \in \cS} \deg(v)} \cdot \prod_{v \in \cS} \left( 1 - \frac{1}{q^{\deg(v)}}\right),
$$ and
$$
E_\cS(\fb,a,\fp)=
\begin{cases}
\left(1+q^{\frac{\fb \rho_\cS-\deg(\fp)}{2}}\right)\fb^{s-1}+(1+h_K(a))\left(q^{\fb \rho_\cS-\deg(\fp)+h_K(a)-R(a)}\right)\fb^{s-2}
+q^{-\deg(\fp)+h_K(a)+1}& \textrm{if } s >1,\\
1+q^{h_K(a)-\deg(\fp)}+q^{\frac{\fb \rho_\cS-\deg(\fp)}{2}}&\textrm{if } s=1,
\end{cases}
$$
where $R(a)$ is defined in \eqref{eq:defRa}.

The constant in the error term depends on $\cS$ (and hence on $K$).
\end{prop}

\begin{proof}
 Assume first that $s > 1$. From \eqref{eq:44},
$$
N^{=}_{h_K}(\mathfrak{p}+a,\mathfrak{b}\rho_\cS)=\sum_{\underline{b} \in \mathbb{B}_\fb^{(0)}} \#(\mathfrak{p}+a)_{=\underline{b}}+\sum_{\underline{b} \in \mathbb{B}_\fb^{(1)}} \#(\mathfrak{p}+a)_{=\underline{b}}.
$$
From $\#\mathbb{B}_\fb=O_{\mathcal{S}}(\fb^{s-1})$  (by Lemma \ref{lem:dedea}) and Corollary \ref{cor:dedea}, we have
$$
\mathbb{B}_{\fb}^{(12)} \leq \mathbb{B}_{\fb}^{(1)}=O_{\mathcal{S}}(\min\{1+h_K(a),\fb \rho_\cS\}\fb^{s-2}).
$$
From Lemma \ref{lem:Bb}, it follows that
\begin{align}\label{eq:55}
\sum_{\underline{b} \in \mathbb{B}_\fb^{(12)}} \#(\mathfrak{p}+a)_{\underline b}=
O_{\cS} \left((1+h_K(a))q^{\fb \rho_\cS-\deg(\fp)+h_K(a)-R(a)} \fb^{s-2}+\fb^{s-1}\right).
\end{align}
\noindent
Using Lemma \ref{lem:Bb11}, 
\begin{equation}\label{eq:cdssd2}
\sum_{\underline{b} \in \mathbb{B}_\fb^{(11)}}\#(\mathfrak{p}+a) \leq \# \mathscr{L}(Q)=O_{\mathcal{S}}\left( 1+q^{-\deg(\fp)+h_K(a)+1} \right).
\end{equation}
\vskip 1mm
\noindent
Using Corollary \ref{cor:444a} and Lemma \ref{cor:22as}, and noting that $\#\mathbb{B}_\fb^{(0)}=O_{\mathcal{S}}(\fb^{s-1})$ (by Lemma \ref{lem:dedea}), we get
$$
\sum_{\underline{b} \in \mathbb{B}_\fb^{(0)}} \#(\mathfrak{p}+a)_{=\underline{b}}=\tilde\kappa_{K,\cS} \fb^{s-1} q^{\fb \rho_\cS-\deg(\fp)-g_K+1}+O_{\cS}(E_1),
$$ where
\begin{align}\label{eq:defE1}
E_1&=\left(1+q^{\frac{\fb \rho_\cS-\deg(\fp)}{2}}\right)\fb^{s-1}+q^{\fb \rho_\cS-\deg(\fp)}(1+h_K(a))(1+\fb^{s-2}).
\end{align}
Clearly \eqref{eq:55}, \eqref{eq:cdssd2} and \eqref{eq:defE1} are absorbed into $O_{\cS}\left(E_{\cS}(\fb,a,\mathfrak{p})\right)$. 
 
 Next assume $s=1$. Then $\mathcal{S}$ consists of a single place $\{w\}$, and 
 $$
 \mathbb{B}_\fb=\{b= \fb\}.
 $$
 Since $\fb >0$, we have $\cS_{b}=\cS$. 
 \vskip 1mm
 \noindent
 $\bullet $ Suppose $b > -\ord_w(a)$, that is $b \in \mathbb{B}_\fb^{(0)}$. By Lemma \ref{cor:22as},
\begin{align}
N^{=}_{h_K}(\mathfrak{p}+a,\mathfrak{b}\rho_\cS)&=\#(\mathfrak{p}+a)_{=b}=\left( \prod_{v \in \cS}\left( 1 - \frac{1}{q^{\deg(v)}}\right) \right) q^{\fb \rho_\cS-\deg(\fp)-g_K+1}+O_{\cS}\left(1+q^{\frac{\fb \rho_\cS-\deg(\fp)}{2}}\right)\\
&=\tilde \kappa_{K,\cS} q^{\fb \rho_\cS-\deg(\fp)-g_K+1}+O_{\cS}\left(1+q^{\frac{\fb \rho_\cS-\deg(\fp)}{2}}\right).
\end{align}
\vskip 1mm
\noindent
$\bullet$ Suppose now that $b \leq -\ord_w(a)$. Then 
$$
\fb \rho_\cS=b\deg(w) \leq - \ord_w(a) \deg(w)=h_K(a).
$$ For $\gamma \in (\mathfrak{p}+a)_{=\fb}$, we have 
\begin{align*}
h_K(\gamma-a)&=\max \{0,-\ord_w((\gamma)-a) \}\cdot \deg(w)\\
&\leq \max\{ 0, -\ord_v(\gamma),-\ord_w(a)\}\cdot \deg(w)\\
&=\max\{0,h_K(\gamma),h_K(a)\} \leq \max\{0,\fb \rho_\cS,h_K(a)\}=h_K(a),
\end{align*} and so $\gamma-a \in \mathfrak{p}_{=h_K(a)}$.
Thus $(\mathfrak{p}+a)_{=\fb}-a \subset \mathfrak{p}_{=h_K(a)}$, and therefore
$$
N^{=}_{h_K}(\mathfrak{p}+a,\mathfrak{b}\rho_\cS)=\#((\mathfrak{p}+a)_{=\fb}-a) \leq \#(\mathfrak{p})_{=h_K(a)}. 
$$ By Lemma \ref{lem:Bb11},
$$
\#(\mathfrak{p})_{=h_K(a)} \leq 1 +q^{h_K(a)-\deg(\fp)+1}=O_{\cS}(1+q^{h_K(a)-\deg(\fp)}).
$$ This completes the proof.
\end{proof}

\subsubsection{Proof of Theorem \ref{cor:funcfiel}} To get an error term that is uniform in $a$, we need to work with a small representative of the coset $a+\mathfrak{p}$. For a nonzero ideal $\mathfrak p\subset \cO_{\cS}$, define
$$
\Psi(\mathfrak p)=\sup_{x \in \cO_{\cS}/\mathfrak p} \min_{\alpha \in \pi^{-1}(x)} h_K(\alpha),
$$
where $
\pi: \cO_{\cS} \longrightarrow \cO_{\cS}/\mathfrak p
$
is the canonical reduction map. Let $Q$ be a nonzero effective divisor with support in $\cS$.

\begin{lem}\label{lem:ringintclas}
Let 
$$
m:= \left\lceil\frac{2g_K-1+ \deg(\mathfrak{p})}{\deg(Q)} \right\rceil.
$$ Then the reduction map $$
\pi: \mathscr{L}(m Q) \subset \cO_\cS \longrightarrow \cO_\cS/\mathfrak p
$$ is surjective. In particular,
\begin{equation}\label{eq:bound1}
\Psi(\mathfrak p) \leq\deg (\mathfrak p)+ 2g_K-1+\deg(Q).
\end{equation}
\end{lem}

\begin{proof}
For $a\in \mathscr{L}(m Q)\subset \cO_\cS$, we have
\begin{equation}\label{eq:lowboundhka}
h_K(a)=
\sum_{v\in \cS}  \max\{0,-\ord_v(a)\} \deg(v)\le m \deg(Q).
\end{equation}
Since $\deg(m Q) \geq \deg(mQ- \fp) \geq 2g_K-1$, by the Riemann--Roch theorem, we have
$$
\#(\mathscr{L}(m Q ))=
q^{\,m \deg(Q)-g_K+1},
$$
and
$$
\#(\mathscr{L}(m Q- \mathfrak{p}))=q^{m\deg(Q)-\deg(\mathfrak{p})-g_K+1},
$$ hence
$$
\#\mathscr{L}(mQ)/\# \mathscr{L}(mQ -\mathfrak{p})=q^{\deg(\mathfrak{p})}.
$$
In particular
$$
\pi: \mathscr{L}(m Q) \subset \cO_\cS \longrightarrow \cO_\cS/\mathfrak p
$$
is surjective; indeed, we have a short exact sequence
$$
0 \rightarrow \mathscr{L}(mQ-\mathfrak{p}) \rightarrow \mathscr{L}(mQ) \rightarrow \cO_\cS/\mathfrak{p}.
$$ 
Therefore, for every $x\in \cO_\cS/\mathfrak p$, there exists $a\in \mathscr{L}(m Q)$ such that $\pi(a)=x$ and
$$
h_K(a)
\le
m\deg(Q)
\le
\deg (\mathfrak p)
+
2 g_K -1+\deg(Q) \quad \textrm{by \eqref{eq:lowboundhka}}.
$$
Taking the supremum over $x$ gives the desired bound \eqref{eq:bound1}.
\end{proof}

\begin{proof}[Proof of Theorem \ref{cor:funcfiel}] Recall that $Q=\sum_{v \in \cS} v$. Let $E_\cS(\fb,a,{\rm N}(\mathfrak{p}))$ be as in Proposition \ref{prop:termcruc}. By Lemma \ref{lem:ringintclas}, we may assume that
$$
h_K(a) \leq \deg(\fp)+2g_K-1+\deg(Q),
$$ and that

$$
\max\{0,-\ord_w(a)\} \leq m Q_w,
$$ for all $w \in \mathcal{S}$, where 
$$
m:= \left\lceil\frac{2g_K-1+ \deg(\mathfrak{p})}{\deg(Q)} \right\rceil.
$$ Thus,
$$
h_K(a)-R(a)\leq m (\deg(Q)-1) \leq \frac{(2g_K-1+\deg(Q))(\deg(Q)-1)}{\deg(Q)}+\left( 1-\frac{1}{\deg(Q)}\right)\deg(\mathfrak{p}),
$$ where $R(a)$ is defined in \eqref{eq:defRa}; in the first inequality, we use the fact that ${\rm supp}(Q)=\cS$. Moreover, $1+q^{-\deg(\fp)+h_K(a)+1}=O_{K,\cS}(1)$. Letting
$$
E_\cS(\fb,\fp):=
\begin{cases}
\left(1+\frac{1}{{\rm N}(\fp)^{1/2}}\cdot q^{\frac{\fb \rho_\cS}{2}}\right)\fb^{s-1}+\frac{q^{\fb \rho_\cS} }{{\rm N}(\fp)^{1/\deg(Q)}}\deg(\fp)\fb^{s-2}& \textrm{if } s >1,\\
1+\frac{1}{{\rm N}(\fp)^{1/2}}\cdot q^{\frac{\fb \rho_\cS}{2}}&\textrm{if } s=1.
\end{cases},
$$ it follows that $E_\cS(\fb,a,\mathfrak{p})=O_{K,\cS}(E_{\cS}(\fb,\fp))$. For $s \geq 2$, since $\deg(Q) \geq 2$, we have
$$
\frac{\fb^{s-1} q^{\fb \rho_\cS/2}}{{\rm N}(\fp)^{1/2}}=O_q\left( \frac{q^{\fb \rho_\cS} \fb^{s-2}\deg(\fp)}{{\rm N}(\fp)^{1/\deg(Q)}}\right).
$$
Therefore, $E_\cS(\fb,\fp)=O_{K,\mathcal{S}}(\tilde E_\cS(\fb,\fp))$,
which completes the proof.
\end{proof}

\section{Around Ekedahl-Poonen and integral hilbertianity results}
\label{sec:Ekehdal-poonen} 
The main results are  Theorem~\ref{thm:denposHilIrrThm} and Theorem~\ref{thm:ekpoongen}. 
Section \ref{thm-more-general} states and proves Theorem \ref{thm:ekpoongen-2}, which is generalization of Theorem \ref{thm:ekpoongen}. 
Theorem \ref{thm:denposHilIrrThm} is then proved in \S \ref{ssec:int-hil-thm}.
The proof of Theorem \ref{thm:ekpoongen-2} involves a lemma (Lemma \ref{lem:technical}), the proof of which is a technical adjustment of an argument from \cite{Poo03}. The detailed proof of Lemma \ref{lem:technical} is postponed to \S \ref{sec:second-estimate}.

\subsection{Generalizing the Ekedahl-Poonen formula} \label{thm-more-general}

\subsubsection{Statement of Theorem \ref{thm:ekpoongen-2}} \label{ssec:thm31-statement}

Let $K$ be a global field with ring of integers $\mathcal{O}_K$ and $\underline T=(T_1,\ldots,T_k)$ be a $k$-tuple of variables ($k\geq 1$). Let $\underline{Q}=\{Q_1, \ldots, Q_\ell\} \subset \mathcal{O}_K[\ul T]$ be a set of $\ell \geq 2$ nonzero polynomials, coprime in $K[\underline{T}]$. 
For every subset $\Sigma$ of nonzero ideals $\mathfrak{a}$ of $\mathcal{O}_K$, set:

$$
\left\{\begin{matrix}
\mathcal{R}_{\Sigma} = \left\{ \ul m \in \mathcal{O}_K^k \mid \emph{no ideal } \mathfrak{a}\in \Sigma \emph{ divides all of } Q_1(\ul m),\ldots,Q_\ell(\ul m) \right\} \hfill \\
\\
\mathcal{C}_{\Sigma}= \mathcal{O}_K^k \setminus \mathcal{R}_{\Sigma} = \big\{ \ul m \in \mathcal{O}_K^k \mid Q_1(\underline{m})\equiv \dots \equiv Q_\ell(\underline{m})\equiv 0 \hskip -3mm \pmod{\mathfrak a}\ \emph{for some ideal } \mathfrak{a}\in \Sigma\big\} \hfill\\
\end{matrix} \right.
$$

\vskip 1,5mm

\noindent 
and merely write $\mathcal{R}_{\mathfrak a}$ and $\mathcal{C}_{\mathfrak a}$ if $\Sigma= \{\mathfrak a\}$. Note that, with ${\mathcal R}$ and ${\mathcal R}^\star$ introduced in Section \ref{sec:intro}, 
\vskip 1,5mm

\centerline{$\displaystyle
\mathcal{R}_{\emptyset} = \mathcal{O}_K^k \hskip 3mm \hbox{\rm and } \hskip 2mm 
\mathcal{R}_ {({\rm Spec}\hskip 1pt \mathcal{O}_K)^\ast} = \mathcal{R}^\star  
$,}
\vskip 1,5mm

\noindent
and so $\mathcal{R}_ {({\rm Spec}\hskip 1pt \mathcal{O}_K)^\ast} = \mathcal{R}^\star= \mathcal{R}$ if in addition $K$ is of class number $1$. 
\vskip 0,5mm
 
 Given a subset $\mathscr{C} \subset ({\rm Spec}\hskip 1pt \mathcal{O}_K)^\ast$, we say that the set $\Sigma$  is \emph{supported on} $\mathscr{C}$ if the prime ideal factorizations of ideals in  $\Sigma$ only involve prime ideals in $\mathscr{C}$. We denote  further the complement $({\rm Spec}\hskip 1pt \mathcal{O}_K)^\ast \setminus \mathscr{C}$
 by ${\mathscr{C}}^c$.

\begin{thm} \label{thm:ekpoongen-2} Assume that $K$ is an arbitrary global field.
Let $\mathscr{C} \subset ({\rm Spec}\hskip 1pt \mathcal{O}_K)^\ast$ be a finite set, and let $\Sigma$ be a finite set of nonzero ideals of $\mathcal{O}_K$ supported on $\mathscr{C}$. Then  

$$
\mu\left({\mathcal R}_{\Sigma} \cap {\mathcal R}_{{\mathscr{C}}^c}\right)=\mu\left({\mathcal R}_{\Sigma}\right) \cdot \prod_{ \mathfrak{p} \in {\mathscr{C}}^c} \left( 1 - \frac{c_{\mathfrak{p}}}{{\rm N}(\mathfrak{p})^k} \right).
$$

\noindent
where the two terms $\mu(\cdot)$ are indeed densities for the sets ${\mathcal R}_{\Sigma}$ and ${\mathcal R}_{\Sigma} \cap {\mathcal R}_{{\mathscr{C}}^c}$, and the infinite 
product is absolutely convergent.
\end{thm}

Theorem \ref{thm:ekpoongen} is the special case of Theorem \ref{thm:ekpoongen-2} for which $\Sigma = \mathscr{C} = \emptyset$. 
If in addition $c_{\mathfrak{p}} < {\rm N}(\mathfrak{p})^k$ for all {prime} ideals $\mathfrak{p}$ of ${\mathcal O}_K$, then ${\mathcal R}^\star$ is a subset of ${\mathcal R}$ of positive density.  The general formula from Theorem \ref{thm:ekpoongen-2} with $\Sigma$ and $\mathscr{C}$ possibly nonempty is used in the proof of Theorem~\ref{thm:denposHilIrrThm} (in \S \ref{ssec:int-hil-thm}): the set ${\mathcal R}$ there satisfies the weaker assumption that $c_{\mathfrak{p}} < {\rm N}(\mathfrak{p})^k$ for all \emph{principal} {prime} ideals $\mathfrak{p}$ of ${\mathcal O}_K$, and will still be shown to contain a subset of positive density.

\subsubsection{Density of $\mathcal{C}_{\mathfrak{p}}$ and $\mathcal{R}_{\mathfrak{p}}$}
\label{ssc:step1} The proof of Theorem \ref{thm:ekpoongen-2} starts with local estimates for $\Sigma = \{\mathfrak{p}\}$.

\begin{lem} \label{local-lemma} Let $\mathfrak{p} \in ({\rm Spec}\hskip 1pt \mathcal{O}_K)^\ast$.
\vskip 1mm

\noindent
{\rm (a)} \hskip 0mm The sets $\mathcal{C}_\mathfrak{p}$ and $\mathcal{R}_\mathfrak{p}$ have densities \hskip 1mm $\displaystyle \mu(\mathcal{C}_\mathfrak{p})= \frac{c_\mathfrak{p}}{{\rm N}(\mathfrak{p})^k} $\hskip 3mm and \hskip 3mm $\displaystyle \mu(\mathcal{R}_\mathfrak{p}) = 1- \frac{c_\mathfrak{p}}{{\rm N}(\mathfrak{p})^k},$
\vskip 1mm

\noindent
{\rm (b)} \hskip 0mm
If $N(\mathfrak{p})$ is of  sufficiently large norm, we have: $
c_\mathfrak{p} \le d^\ell \cdot {\rm N}(\mathfrak{p})^{k-2}
$, where $d=\max_{1\leq i\leq \ell} \deg(Q_i)$. 
\end{lem}

Assertion (b) guarantees that the infinite product from Theorem \ref{thm:ekpoongen-2} is absolutely convergent.

\begin{proof} (a) 
Denote by $\overline{\mathcal{C} _{\mathfrak{p}}}$ the set of cosets of elements of $\mathcal{C} _{\mathfrak{p}}$ modulo the ideal $\mathfrak{p}$. Note that $\# \overline{\mathcal{C} _{\mathfrak{p}}}$ is the cardinality denoted by $c_{\mathfrak{p}}$ in Section \ref{sec:intro}. Let  $\widetilde{\mathcal C}_\mathfrak{p}$ be a system of representatives in $\mathcal{O}_K^k$ of $\overline{{\mathcal C}_\mathfrak{p}}$ under the reduction modulo $\mathfrak{p}$. By definition of ${\mathcal C}_\mathfrak{p}$, we have:
\begin{align}
\mathcal{C}_\mathfrak{p} \cap \mathcal{O}^k_{K,B}&=\left\lbrace \underline{m} \in \mathcal{O}^k_{K,B} \mid \underline{m} \equiv \underline{a} \hskip -2mm \pmod{\mathfrak{p}}\,\,\textrm{for some}\,\, \underline{a} \in \widetilde{C}_{\mathfrak{p}}\right\rbrace\\
&=\bigsqcup_{\underline{a} \in \widetilde{C}_\mathfrak{p}}\prod_{i=1}^k\left\lbrace m_i \in \mathcal{O}_K \mid m_i \equiv a_i \hskip -2mm \pmod{\mathfrak{p}}\,\,\textrm{and}\,\,{\rm H}_K(m_i) \leq B \right\rbrace \\
&= \bigsqcup_{\underline{a} \in \widetilde{C}_\mathfrak{p}} \prod_{i=1}^k \hskip 1mm (\mathfrak{p}+a_i)_B 
\end{align}
Conclusion (a) then follows from conjoining Definition \ref{defn:density} and Proposition~\ref{lem:bounheigharithprog}.

\vskip 2mm

\noindent
(b) For completeness, we include the argument already given in \cite[Corollary~6.1]{BD23}. 

Assume first $k=1$. By coprimality of $Q_1,\ldots, Q_\ell$ in $K[T]$, there exists a nonzero
$$
\delta \in \left(\sum_{i=1}^\ell\mathcal{O}_K[T] Q_i(T)\right) \cap \mathcal{O}_K.
$$ 

\noindent
This proves that $c_\mathfrak{p} = 0$ for all $\mathfrak{p} \in {\rm Spec}(\mathcal{O}_K)$ for which ${\rm N}(\mathfrak{p}) > |{\rm N}(\delta)|$.
\vskip 1mm

Assume now that $k\ge 2$. 
Let $\mathfrak{p} \in ({\rm Spec}\hskip 1pt \mathcal{O}_K)^\ast$. 
Set $q={\rm N}(\mathfrak{p})$.
Since $Q_1,\ldots, Q_\ell$ are coprime in $K[\ul T]$, it follows from \cite[Prop.~3.2]{BD23} that their reductions modulo $\mathfrak{p}$, namely $\overline{Q_1}^{\mathfrak{p}},\ldots, \overline{Q_\ell}^{\mathfrak{p}}$, are nonzero and remain coprime in $\mathbb{F}_q[\ul T]$ for all sufficiently large $q$. Consequently, they are also coprime in $\overline{\mathbb{F}_q}[\ul T]$ for the same primes $\mathfrak{p}$; see, for example, \cite[\S 2.1]{BDN20b}.

Fix such a prime ideal $\mathfrak{p}$, and consider the ideal $\mathcal{I} = \langle \overline{Q_1}^{\mathfrak{p}},\ldots, \overline{Q_\ell}^\mathfrak{p}\rangle \subset \overline{\mathbb{F}_q}[\ul T]$. We aim to apply the B\'ezout theorem, in the following form, due to Lachaud--Rolland \cite[Corollary~2.2]{LR15}: 
\vskip 1mm

\centerline{$c_\mathfrak{p} \leq d^\ell \cdot q^{\dim(Z(\mathcal{I}))}$,}
\vskip 1mm

\noindent
and so are reduced to show that $\dim(Z(\mathcal{I})) \leq k-2$. 

It is well-known that the dimension $\dim(Z(\mathcal{I}))$ coincides with the Krull dimension $\dim  \overline{\mathbb{F}_q}[\ul T]/\mathcal{I}$ of the quotient ring $\overline{\mathbb{F}_q}[\ul T]/\mathcal{I}$ (see \hbox{e.g.} \cite[Proposition~1.7]{Har77}). 
We may assume that $\deg \overline{Q_1}^{\mathfrak{p}} \geq 1$: otherwise $c_\mathfrak{p}=0$. Then $\overline{Q_1}^{\mathfrak{p}}$ has at least one irreducible factor $\Delta \in \overline{\mathbb{F}_q}[\ul T]$. 
We have:
\vskip 1mm
\centerline{$\dim {\overline{\mathbb{F}_q}[\ul T]/\mathcal{I}} \leq \dim{\overline{\mathbb{F}_q}[\ul T]/\langle \overline{Q_1}^{\mathfrak{p}}\rangle} = \dim{\overline{\mathbb{F}_q}[\ul T]/\langle \Delta\rangle} = k-1$.}

\vskip 2mm

Assume on the contrary that $\dim {\overline{\mathbb{F}_q}[\ul T]/\mathcal{I}} = k-1$. Thus there exists a minimal prime ideal ${\mathfrak{q}} \subset \overline{\mathbb{F}_q}[\ul T]$ containing $\mathcal{I}$ such that $\dim {\overline{\mathbb{F}_q}[\ul T]/{\mathfrak{q}}} = k-1$, or, equivalently ${\mathfrak{q}}$ is of height $1$. 
By Krull's Hauptidealsatz \cite[Theorem~1.11 A \& Proposition~1.13]{Har77}, the variety $Z({\mathfrak{q}})$ is a hypersurface $Z(f)$, for some irreducible polynomial $f\in \overline{\mathbb{F}_q}[\ul T]$. But then, since $ \langle f \rangle = {\mathfrak{q}} \supset \mathcal{I}$, it follows that $f$ divides each polynomial $\overline{Q_i}^\mathfrak{p}$ in $\overline{\mathbb{F}_q}[\ul T]$, ($1 \leq i \leq \ell$), a contradiction. We conclude that $\dim(Z(\mathcal{I}))=\dim {\overline{\mathbb{F}_q}[\ul T]/\mathcal{I}} \leq k-2$.
\end{proof}

\subsubsection{A lemma of Poonen} \label{ssec:conditional}
For  $M \geq 1$, define 

$$\Gamma_{\leq M}=\{\mathfrak{p} \in ({\rm Spec}\hskip 1pt \mathcal{O}_K)^\ast \mid {\rm N}(\mathfrak{p}) \leq M\}$$

\vskip 1,5mm

Let $\mathscr{C} \subset ({\rm Spec}\hskip 1pt \mathcal{O}_K)^\ast$ be a finite set. Define $\mathcal{C}_{> M}$ to be the set $\mathcal{C}_{{\mathscr{C}}^c \setminus \Gamma_{\leq M}}$, that is:

\vskip 1mm

\centerline{%
\(\displaystyle
\mathcal{C}_{> M} = \big\{ \underline{m} \in \mathcal{O}_K^k \mid Q_1(\underline{m}) \equiv \dots \equiv Q_s(\underline{m}) \equiv 0 \hskip -1mm \pmod{\mathfrak{a}} \text{ for some ideal } \mathfrak{a} \notin (\Gamma_{\leq M} \cup \mathscr{C}) \big\}
\).
}
\vskip 1.5mm

\begin{lem} \label{lem:technical}
We have: \hskip 5mm $\displaystyle \lim_{M \rightarrow \infty} \mu_{{\rm sup}}(\mathcal{C}_{>M})=0$.
\end{lem}

\begin{proof} The proof is similar to that of \cite[Lemma~5.1]{Poo03} and is postponed to \S \ref{sec:second-estimate}.\end{proof}

\vskip 2mm

\subsubsection{A technical intermediate stage} For $M \geq 1$, define $\mathcal{R}_{\le M}$ to be the set ${\mathcal R}_{\Gamma_{\leq M}\setminus \mathscr{C}}$, that is:

 $$\mathcal{R}_{\le M} = \big\{ \ul m \in \mathcal{O}_K^k \mid \textit{ no ideal } \mathfrak{p} \in \Gamma_{\leq M}\setminus \mathscr{C}\,\,\textit{divides all of}\,\, Q_1(\ul m),\ldots,Q_\ell(\ul m)  \big\}.$$
\vskip 1mm

 The following result is a technical intermediate stage towards the proof of Theorem \ref{thm:ekpoongen-2}.

\begin{prop} \label{main-lem}
 Let $\underline{m}_0 \in \mathcal{O}_K^k$ and let $\Psi\subset {\mathcal O}_K$ be a nonzero ideal supported on $\mathscr{C}$. Then the sets 
$\mathcal{R}_{\leq M} \cap(\underline{m}_0 + \Psi^k)$ and ${\mathcal R}_{{\mathscr{C}}^c} \cap (\underline{m}_0+\Psi^k)$ have a density and\vskip 1,5mm
 
\noindent
{\rm (a)} \hskip 5mm $\displaystyle \mu(\mathcal{R}_{\leq M} \cap(\underline{m}_0 + \Psi^k)) =\frac{1}{{\rm N}(\Psi)^k}\prod_{\mathfrak{p} \in \Gamma_{\leq M}\setminus \mathscr{C}} \mu(\mathcal{R}_{\mathfrak{p}})$.
 
\vskip 1mm
 
\noindent
{\rm (b)} \hskip 5mm 
 $\displaystyle \mu({\mathcal R}_{{\mathscr{C}}^c} \cap (\underline{m}_0+\Psi^k))=\lim_{M \rightarrow \infty}\mu(\mathcal{R}_{\le M} \cap (\underline{m}_0+\Psi^k)).$

\end{prop}

\begin{proof}
(a) Set $\Gamma_{\leq M}\setminus \mathscr{C} = \{\mathfrak{p}_1,\ldots,\mathfrak{p}_e\}$.
By the Chinese Remainder Theorem, modulo $\theta=\mathfrak{p}_1 \cdots \mathfrak{p}_e \Psi$, the cardinality of $\mathcal{R}_{\leq M} \cap (\underline{m}_0 + \Psi^k)$ is equal to $({\rm N}(\mathfrak{p}_1)^k-c_{\mathfrak{p}_1})\cdots ({\rm N}(\mathfrak{p}_e)^k-c_{\mathfrak{p}_e})$. By Proposition~\ref{lem:bounheigharithprog} and Lemma \ref{local-lemma}, we have

\begin{align*}\label{eq:muofRleqM}
\mu(\mathcal{R}_{\leq M} \cap(\underline{m}_0 + \Psi^k))&=\frac{({\rm N}(\mathfrak{p}_1)^k-c_{\mathfrak{p}_1})\cdots ({\rm N}(\mathfrak{p}_e)^k-c_{\mathfrak{p}_e})}{{\rm N}(\mathfrak{p}_1 \cdots \mathfrak{p}_e \Psi)^k}\\
&=\frac{1}{{\rm N}(\Psi)^k}\prod_{j=1}^e \left( 1 -\frac{c_{\mathfrak{p}_j}}{{\rm N}(\mathfrak{p}_j)^k} \right)=\frac{1}{{\rm N}(\Psi)^k}\prod_{j=1}^e \mu(\mathcal{R}_{\mathfrak{p}_j}).
\end{align*}

\vskip 2mm

\noindent 
(b) Note first that, as ${\Gamma_{\leq M}\setminus \mathscr{C}}\subset \mathscr{C}^c$, we have ${\mathcal R}_{{\mathscr{C}}^c} \subset \mathcal{R}_{\le M}$. We then have $\mathcal{R}_{\le M} \setminus {\mathcal R}_{{\mathscr{C}}^c}
 \subset \mathcal{C}_{> M}$: indeed, elements $\underline m$ of the set $\mathcal{R}_{\le M} \setminus {\mathcal R}_{{\mathscr{C}}^c}$, not being in ${\mathcal R}_{{\mathscr{C}}^c}$ have the property that there exists a prime ideal $\mathfrak{p} \in \mathscr{C}^c$ that divides all the $Q_i(\ul m)$; but as these elements $\underline m$ are in $\mathcal{R}_{\le M}$, the ideal $\mathfrak{p}$ must be of norm ${\rm N}(\frak p) > M$; hence $\underline m \in  \mathcal{C}_{> M}$. Thus,
$$\mathcal{R}_{\le M} 
= {\mathcal R}_{{\mathscr{C}}^c} \cup (\mathcal{R}_{\le M} \setminus {\mathcal R}_{{\mathscr{C}}^c})
\subset {\mathcal R}_{{\mathscr{C}}^c} \cup \mathcal{C}_{> M}.
$$ This inclusion implies that
$$\mu_{{\rm sup}}({\mathcal R}_{{\mathscr{C}}^c} \cap (\underline{m}_0+\Psi^k)) \le \mu(\mathcal{R}_{\le M} \cap (\underline{m}_0+\Psi^k) ) \le \mu_{{\rm sup}}({\mathcal R}_{{\mathscr{C}}^c} \cap (\underline{m}_0+\Psi^k)) + \mu_{{\rm sup}}(\mathcal{C}_{> M} \cap (\underline{m}_0+\Psi^k)),$$ and
$$\mu_{{\rm inf}}({\mathcal R}_{{\mathscr{C}}^c} \cap (\underline{m}_0+\Psi^k)) \le \mu(\mathcal{R}_{\le M} \cap (\underline{m}_0+\Psi^k) ) \le \mu_{{\rm inf}}({\mathcal R}_{{\mathscr{C}}^c} \cap (\underline{m}_0+\Psi^k)) + \mu_{{\rm sup}}(\mathcal{C}_{> M} \cap (\underline{m}_0+\Psi^k)).$$
By Lemma \ref{lem:technical}, we know that $\lim_{M \rightarrow \infty }\mu_{{\rm sup}}(\mathcal{C}_{>M})= 0$, so
$$ \mu_{{\rm inf}}({\mathcal R}_{{\mathscr{C}}^c} \cap (\underline{m}_0+\Psi^k))=\lim_{M \rightarrow \infty}\mu(\mathcal{R}_{\le M} \cap (\underline{m}_0+\Psi^k)) = \mu_{{\rm sup}}({\mathcal R}_{{\mathscr{C}}^c} \cap (\underline{m}_0+\Psi^k)),
$$
hence
$$
\lim_{M \rightarrow \infty}\mu(\mathcal{R}_{\le M} \cap (\underline{m}_0+\Psi^k)) = \mu({\mathcal R}_{{\mathscr{C}}^c} \cap (\underline{m}_0+\Psi^k)).
$$
\end{proof}

\subsubsection{End of proof of Theorem \ref{thm:ekpoongen-2}}
Let $\Sigma=\{\mathfrak{a}_1,\dots,\mathfrak{a}_\sigma\}$ ($\sigma \geq 0$) be a finite set of nonzero ideals of $\mathcal{O}_K$ supported on $\mathscr{C}$. Let $\Psi$ be the least common multiple of the ideals $\mathfrak{a}_1,\dots,\mathfrak{a}_\sigma$. By definition, the support of $\Psi$ lies in $\mathscr{C}$. 

Writing conclusion (a) of Proposition \ref{main-lem} first for this ideal $\Psi$ and an arbitrary $\underline m_0\in {\mathcal O}_K^k$ 
and then in the special case that ${\underline m}_0=\underline 0$ and $\Psi=\mathcal{O}_K$,
we obtain:

$$
\mu(\mathcal{R}_{\le M} \cap (\underline{m}_0+\Psi^k))= \frac{1}{{\rm N}  (\Psi)^k} \hskip 1mm \mu(\mathcal{R}_{\le M}).
$$

\noindent
Conclusion (b) in the same two cases then yields:

$$
\mu({\mathcal R}_{{\mathscr{C}}^c} \cap (\underline{m}_0+\Psi^k))= \frac{1}{{\rm N}(\Psi)^k}\hskip 1mm\mu({\mathcal R}_{{\mathscr{C}}^c}).
$$

Observe next that ${\mathcal R}_{\Sigma}$ is invariant under the translation by $\Psi^k$. Hence there exist finitely many $\underline{m}_1,\dots,\underline{m}_v \in \mathcal{O}_K^k$ such that
$$
{\mathcal R}_{\Sigma}=\bigcup_{i=1}^v \{\underline{m}_i+\Psi^k\}.
$$ 

\noindent
We obtain that the sets ${\mathcal R}_{\Sigma}$ and ${\mathcal R}_{\Sigma} \cap {\mathcal R}_{{\mathscr{C}}^c}$ have a density and that
$$
\mu({{\mathcal R}_{{\mathscr{C}}^c} \cap \mathcal R}_{\Sigma})=\sum_{i=1}^v \mu({\mathcal R}_{{\mathscr{C}}^c} \cap (\underline{m}_i+\Psi^k))=\sum_{i=1}^v \frac{1}{{\rm N}(\Psi)^k}\mu({\mathcal R}_{{\mathscr{C}}^c})=\frac{v}{{\rm N}(\Psi)^k} \cdot \mu({\mathcal R}_{{\mathscr{C}}^c})=\mu({\mathcal R}_{\Sigma})\cdot \mu({\mathcal R}_{{\mathscr{C}}^c}),
$$ 

\noindent
the equality ${v}/{{\rm N}(\Psi)^k} =\mu({\mathcal R}_{\Sigma})$ coming from Proposition \ref{prop:crucari}. 

\noindent
Finally re-use Proposition \ref{main-lem} part (b) followed by part (a), for ${\underline m}_0=\underline 0$ and $\Psi=\mathcal{O}_K$, and finally conjoin it with Lemma \ref{local-lemma} to express the last term $\mu({\mathcal R}_{{\mathscr{C}}^c})$ above as follows:

$$
\mu({\mathcal R}_{{\mathscr{C}}^c})=\lim_{M \rightarrow \infty}\mu(\mathcal{R}_{\le M}) = \prod_{\mathfrak{p} \in {\mathscr{C}}^c}\mu(\mathcal{R}_{\mathfrak{p}}) = \prod_{\mathfrak{p} \in {\mathscr{C}}^c}
\big(1- \frac{c_\mathfrak{p}}{{\rm N}(\mathfrak{p})^k}\big).
$$

\subsection{Proof of Theorem~\ref{thm:denposHilIrrThm}} \label{ssec:int-hil-thm} 
Let $P_1,\ldots, P_s$ be polynomials as in the statement of Theorem \ref{thm:denposHilIrrThm}. Let $P= P_1\cdots P_s\in {\mathcal O}_K[\underline T,\underline Y]$ be the product of all of them.
Let $\underline{Q}=\{Q_{1},.\dots,Q_{\ell}\} \subset \mathcal{O}_K[\underline{T}]$ be the set of coefficients of $P$, viewed as a polynomial in $\underline{Y}$. Since for each $i=1,\ldots, s$, the polynomial $P_i$ is irreducible in $K[\underline T,\underline Y]$, the coefficients of $P_i$, as a polynomial in $\underline Y$, are coprime in $K[\underline T]$. By Gauss's lemma, so are the coefficients of $P$, \hbox{i.e.} the elements of $\underline Q$.
Consider the set ${\mathcal R} \subset {\mathcal O}_K^k$ attached to the set $\underline Q$.

The special case $\ell = 1$ should be dealt with separately. In this case, the product $P = P_1\cdots P_s$ is a monomial in $\underline Y$ with coefficient in ${\mathcal O}_K[\underline T]$. Then so are the polynomials $P_1,\ldots, P_s$, with the additional condition that the coefficients should be units in ${\mathcal O}_K$, due to the assumptions that $P_1,\ldots, P_s$ are irreducible $K[\underline T,\underline Y]$ and that $P$ has no fixed divisor \hbox{w.r.t.} $\underline T$. The conclusion of Theorem~\ref{thm:denposHilIrrThm} then easily follows.

In the rest of the proof, we assume that $\ell \geq 2$.
As in \cite[\S 2.1.2]{BD25}, denote the set of all ideals of ${\mathcal O}_K$ generated by irreducible elements by $\irr$. Let $\mathscr{C}$ be the set of all \emph{fixed prime ideal divisors of ${P}$ w.r.t. $\underline{T}$}, \hbox{i.e.}, of all ${\frak p}\in ({\rm Spec}\hskip 1pt \mathcal{O}_K)^\ast$ such that $P(\underline t,\underline Y)\equiv 0 \pmod{\frak p}$\footnote{The polynomial $P$ is assumed to have no fixed prime divisors (in ${\mathcal O}_K$) \hbox{w.r.t.} $\underline{T}$ but may have some fixed prime ideal divisor.}. From \cite[Lemma 3.8 (1)]{BD25} (conjoined with Example 3.4 there), the set $\mathscr{C}$ is finite. Let $\Sigma$ be the set of all ideals in $\irr$ whose support lies in $\mathscr{C}$. From \cite[Lemma 3.9 (1)]{BD25}, the set $\Sigma$ is finite: set $\Sigma = \{\mathfrak{a}_1,\dots,\mathfrak{a}_\sigma\}$.

\begin{lem}\label{lem:incluRR}
With notation from \S \ref{ssec:thm31-statement} and under the assumptions of Theorem~\ref{thm:denposHilIrrThm}, we have:
\vskip 1mm

\noindent
{\rm (a)} ${\mathcal R}_{\Sigma} \cap {\mathcal R}_{{\mathscr{C}}^c} \subset \mathcal{R}
$
\vskip 1mm

\noindent
{\rm (b)} $\mu({\mathcal R}_\Sigma) >0$,
\vskip 1mm

\noindent
{\rm (c)} the set ${\mathcal R}$ contains a subset of positive density.

\end{lem}
\begin{proof} 
(a) Let $\underline{m} \in {\mathcal R}_{\Sigma} \cap {\mathcal R}_{{\mathscr{C}}^c}$. If $\underline m \notin {\mathcal R}$, there exists a proper principal ideal  $\mathfrak{p}\subset {\mathcal O}_K$ such that $\underline{m} \notin \mathcal{R}_{\mathfrak{p}}$. 
Since ${\mathcal O}_K$ is Noetherian, one may and will assume that $\mathfrak{p} \in \irr$.  We have $\underline{P}(\underline{m})\equiv 0 \pmod{\mathfrak{p}}$, so $\underline{P}(\underline{m})\equiv 0$ modulo any prime ideal of ${\mathcal O}_K$ involved in the prime ideal factorization of $\mathfrak{p}$. As $\underline{m} \in {\mathcal R}_{{\mathscr{C}}^c}$, we have ${\rm Supp}(\mathfrak{p}) \subset \mathscr{C}$. Hence $\mathfrak{p}\in \Sigma$, \hbox{i.e.}, $\mathfrak{p}=\mathfrak{a}_j$ for some $j=1,\dots,\sigma$. It follows that $\underline{P}(\underline{m}) \equiv 0 \pmod{\mathfrak{a}_j}$, that is, $\underline{m} \notin \mathcal{R}_{\mathfrak{a}_j}$, a contradiction since $\underline{m} \in {\mathcal R}_\Sigma$.
This proves that ${\mathcal R}_{\Sigma} \cap {\mathcal R}_{{\mathscr{C}}^c} \subset \mathcal{R}$.
\vskip 1mm

\noindent
(b) It follows from the assumption that the product $P=P_1 \cdots P_s$ has no {fixed divisor \hbox{w.r.t.} $\underline{T}$} that $\mathcal{R}_{\mathfrak{a}_i}$ is nonempty for all indices $i=1,\ldots,\sigma$. From \cite[Lemma~3.11]{BD25}, the Dedekind ring ${\mathcal O}_K$ satisfies the ``nonvanishing approximation property \hbox{w.r.t.} $\irr$''. This property guarantees that the intersection ${\mathcal R}_{\Sigma}=\bigcap_{i=1}^\sigma \mathcal{R}_{\mathfrak{a}_i}$ is nonempty \cite[Definition 3.3 (1)]{BD25}. Fix $\underline{m}=(m_1,\ldots,m_k) \in {\mathcal R}_{\Sigma}$. Then the set ${\mathcal R}_{\Sigma}$ contains the product
\vskip 2mm

\centerline{$\prod_{i=1}^{k} \left( m_i + \mathfrak{a}_1 \cdots \mathfrak{a}_\sigma \right)
$,}
\vskip 2mm

\noindent
which, by Proposition~\ref{lem:bounheigharithprog}, has density $\displaystyle 1/{\rm N}(\mathfrak{a}_1 \cdots \mathfrak{a}_\sigma)^k>0$.
\vskip 2mm

\noindent
(c) From (a), it suffices to show that $\mu({\mathcal R}_{\Sigma} \cap {\mathcal R}_{{\mathscr{C}}^c})>0$. Now Theorem \ref{thm:ekpoongen-2} gives:
\vskip 1,5mm

\centerline{$\displaystyle 
\mu({\mathcal R}_{\Sigma} \cap {\mathcal R}_{{\mathscr{C}}^c})=\mu\left({\mathcal R}_{\Sigma}\right) \cdot \prod_{ \mathfrak{p} \in ({\rm Spec}\hskip 1pt \mathcal{O}_K)^\ast \setminus \mathscr{C}} \left( 1 - \frac{c_{\mathfrak{p}}}{{\rm N}(\mathfrak{p})^k} \right).
$}
\vskip 2mm

\noindent
By definition of $\mathscr{C}$, we have 
$\displaystyle 1-(c_{{\frak p}}/{{\rm N}(\mathfrak{p})^k})>0$ for any  $\mathfrak{p} \in  {\mathscr{C}}^c$. Combining this with Lemma \ref{lem:incluRR}(b), it follows that $\mu({\mathcal R}_{\Sigma} \cap {\mathcal R}_{{\mathscr{C}}^c})>0$.
\end{proof}

The proof of Theorem~\ref{thm:denposHilIrrThm} can be completed. Denote the classical Hilbert subset of $K^k$ of all $\underline{m}$ such that $P_1(\underline{m},\underline{Y}), \ldots, P_s(\underline{m},\underline{Y})$ are irreducible in $K[\underline{Y}]$ by $\mathcal{H}_K(P_1,\ldots,P_s)$. By Definition \ref{def:int-Hil}, we have 
\vskip 1mm

\centerline{${\mathcal H}_{\mathcal Z}(P_1,\ldots,P_s) = \mathcal{H}_K(P_1,\ldots,P_s) \cap {\mathcal R}$.}
\vskip 2mm

\noindent
By \cite[Theorem~1.1]{PS24}, we have $\mu(\mathcal{H}_K(P_1,\ldots,P_s)\cap \mathcal{O}_K^k)=1$. Conjoining this with Lemma \ref{lem:incluRR}(c), we obtain that $\mathcal{H}_K(P_1,\ldots,P_s)\cap \mathcal{O}_K^k \cap \mathcal{R}={\mathcal H}_{\mathcal Z}(P_1,\ldots,P_s)$ contains a subset of positive density.

\section{Proof of Lemma \ref{lem:technical}}
\label{sec:second-estimate}

Recall that the goal of this section is to give a proof of the fact that $\lim_{M \rightarrow \infty} \mu_{{\rm sup}}(\mathcal{C}_{>M})=0$ (Lemma \ref{lem:technical}) by adjusting Poonen's arguments given in the case ${\mathcal O}_K={\Zz}$.

Let $\varepsilon<1/\fd$. Fix $M$ and $B \ge M$ such that $\theta(B):=\theta_{\varepsilon}(B)=\log_q(B)^{\frac{1}{1-\varepsilon}} \geq M$ (as in Corollary \ref{lem:boundinarithproBC}) and consider the decomposition:
	$$\mathcal{C}_{>M} = \mathcal{C}_{>M, \le \theta(B)} \cup \mathcal{C}_{>\theta(B)},$$
	where $\mathcal{C}_{>M, \le \theta(B)} = \mathcal{C}_{\Gamma_{\leq \theta(B)} \setminus (\Gamma_{\leq M} \cup \mathscr{C})}$.
	
    \subsection{{Estimate of $\mathcal{C}_{>M, \le \theta(B)}$.}} Note that for sufficiently large $B$ and for each $\mathfrak{p} \in \Gamma_{\leq \theta(B)}\setminus (\Gamma_{\leq M} \cup \mathscr{C})$, we have 
    $$
    \begin{cases}{\rm N}(\mathfrak{p}) \leq {\rm N}(\mathfrak{p})^{1/\fd}\log_q(B)&\textrm{if }K \textrm{ is a number field};\\
    \frac{B}{{\rm N}(\fp)} \geq \max \left\lbrace \log_q(B), \frac{B^{1/2}}{{\rm N}(\fp)^{1/2}} \right\rbrace&\textrm{if }K \textrm{ is a function field and }\#\cS_\infty=1;\\
    \frac{\log_q(B)^{s-1}B}{{\rm N}(\fp)} \geq \max \left\lbrace \log_q(B)^{s},\frac{B \log_q(B)^{s-2} \deg(\fp)}{{\rm N}(\fp)^{1/\fd}}\right\rbrace&\textrm{if }K \textrm{ is a function field and }\#\cS_\infty \geq 2.
    \end{cases}
    $$ Thus, by Theorem \ref{prop:crucari}, there exists a constant $A_K$ (only depending on $K$) such that
    $$
    \frac{\#(\mathfrak{p}+a)_B}{\#\mathcal{O}_{K,B}} \leq \frac{A_K}{{\rm N}(\mathfrak{p})},
    $$ for all $\mathfrak{p} \in \Gamma_{\leq \theta(B)}\setminus (\Gamma_{\leq M} \cup \mathscr{C})$ and $a \in \mathcal{O}_K$. Set $C=(A_K)^k d^\ell$, where $d=\max_{1 \leq i \leq  \ell}\deg(Q_i)$. 
    
    Fix $\mathfrak{p} \in \Gamma_{\leq \theta(B)}\setminus (\Gamma_{\leq M} \cup \mathscr{C})$. From Lemma \ref{local-lemma}(b), we have $c_{\mathfrak{p}} \leq d^\ell {\rm N}(\mathfrak{p})^{k-2}$. Hence,
 $$ \frac{\#(\mathcal{C}_{\mathfrak{p}} \cap \mathcal{O}^k_{K,B})}{\#\mathcal{O}^k_{K,B}} \le c_{\mathfrak{p}} \left( \frac{A_K}{{\rm N}(\mathfrak{p})} \right)^k \le C \cdot \frac{1}{{\rm N}(\mathfrak{p})^2}.
 $$
    
    Consequently, we obtain:
	\begin{equation}
		\label{eq:QMB}
		\frac{\# (\mathcal{C}_{>M, \leq \theta(B)} \cap \mathcal{O}^k_{K,B})}{\# \mathcal{O}^k_{K,B}} 
		\le \sum_{\mathfrak{p} \in \Gamma_{\leq \theta(B)} \setminus (\Gamma_{\leq M} \cup \mathscr{C})} \frac{\# (\mathcal{C}_{\mathfrak{p}} \cap \mathcal{O}^k_{K,B})}{\#\mathcal{O}^k_{K,B}} 
		\le  \sum_{\mathfrak{p} \in \Gamma_{\leq \theta(B)} \setminus (\Gamma_{\leq M} \cup \mathscr{C})} \frac{C}{{\rm N}(\mathfrak{p})^2} 
		\le C \sum_{\mathfrak{p}\notin \Gamma_{\leq M}} \frac{1}{{\rm N}(\mathfrak{p})^2}. 
	\end{equation}
	The last sum is independent of  $B$ and tends to $0$ as $M \to +\infty$.

\subsection{Estimate of $\mathcal{C}_{>\theta(B)}$} Our aim is to prove that:
    \begin{equation}
		\label{eq:QB}	
		\lim_{M \rightarrow \infty} \frac{\# (\mathcal{C}_{>\theta(B)} \cap \mathcal{O}^k_{K,B})}{\# \mathcal{O}^k_{K,B}} =0.
	\end{equation}

 \noindent   
    \emph{Preliminaries.}
	First, we may reduce to the case where each polynomial $Q_i$ is irreducible in $\mathcal{O}_K[\ul T]$. Indeed, suppose that $Q_1 = Q \cdot R$ with $Q,R \in \mathcal{O}_K[\ul T]$. For $\underline{m} \in \mathcal{O}_K^k$ and a prime ideal $ \mathfrak{p}$, if $\mathfrak{p}|Q_1(\ul m)$, then either $ \mathfrak{p} | Q(\ul m)$ or $ \mathfrak{p} | R(\ul m)$. Therefore, $\mathcal{C}_{>\theta(B)}(QR,Q_2,\ldots,Q_\ell) \subset \mathcal{C}_{>\theta(B)}(Q,Q_2,\ldots,Q_\ell) \cup \mathcal{C}_{>\theta(B)}(R,Q_2,\ldots,Q_\ell)$.

	Second, we may also reduce to the case where one of the polynomial -- say, $Q_\ell$ -- depends only on the variables $T_1,\ldots,T_{k-1}$. If this is not the case, then since $Q_1,\ldots,Q_\ell$ are coprime, there exists a B\'{e}zout identity 
    $$\sum_{i=1}^\ell A_i(\ul T)Q_i(\ul T) = \Delta(T_1,\ldots,T_{k-1})
    $$
	where $A_i \in \mathcal{O}_K[\ul T]$($i=1,\ldots,\ell$) and $\Delta \in \mathcal{O}_K[T_1,\ldots,T_{k-1}]$, with $\Delta \neq 0$. For $\underline{m} \in \mathcal{O}_K^k$ and a prime ideal $\mathfrak{p}$, if $\mathfrak{p}$ divides all $Q_i(\ul m)$ ($i=1,\ldots,\ell$), then $ \mathfrak{p} | \Delta(\ul m)$. Hence
	$\mathcal{C}_{>\theta(B)}(Q_1,\ldots,Q_\ell) \subset \mathcal{C}_{>\theta(B)}(Q_1,\ldots,Q_{\ell-1},\Delta)$
	and $Q_1,\ldots,Q_{\ell-1},\Delta$ are coprime polynomials (if some nonconstant polynomial $R$ divides $Q_1$, a polynomial where all the variables $T_1,\ldots, T_k$ occur, then $R=Q_1$, up to some multiplicative constant, because $Q_1$ is irreducible, but then $R$ cannot divide $\Delta$, a polynomial in the variables $T_1,\ldots,T_{k-1}$ only). 

	Third, we may assume that the leading coefficient of $Q_i(\ul T)$, ($1 \leq i\leq \ell-1$), seen as a polynomial in $T_k$, is not divisible by the last polynomial $Q_\ell(T_1,\ldots,T_{k-1})$.
	Indeed, write $Q_i(\ul T) = Q_i^{0}(T_1,\ldots,T_{k-1}) T_k^{\delta_i} + \cdots \in \mathcal{O}_K[T_1,\ldots,T_{k-1}][T_k]$. If $Q_i^0 =F_i Q_\ell$, then we can replace $Q_i$ by $\widetilde{Q}_i = Q_i - F_i Q_\ell T_k^{\delta_i}$, which has strictly lower degree in $T_k$. We repeat this process inductively until either $\deg_{T_k}(Q_i)=0$ or $Q_\ell$ no longer divides $Q_i^0$. Note that the set $\mathcal{C}_{\mathfrak{p}}$ is preserved in this process:
	$\mathfrak{p} | Q_i(\ul m)$ and $Q_\ell(\ul m)$ iff $\mathfrak{p}$ divides $(Q_i - F_i Q_\ell T_k^{\delta_i})(\ul m)$ and $Q_\ell(\ul m)$. If necessary, we may then apply the first reduction to this new set of polynomials, to get irreducible polynomials.
\vskip 1mm
	
	We now prove \eqref{eq:QB}.
\vskip 1mm
\noindent
	\emph{Induction.} The proof proceeds by induction on the dimension $k$.
	For $k=1$, a B\'ezout identity 
    $$\sum_{i=1}^\ell A_i(T)Q_i(T) = \Delta,
    $$ with $A_i \in \mathcal{O}_K[X]$, $\Delta \in \mathcal{O}_K$ and $\Delta \neq 0$, implies that if $\mathfrak{p} | Q_i(m)$ for every $i=1,\ldots,\ell$, then $\mathfrak{p} | \Delta \in \mathcal{O}_K$. Therefore, when $\theta(B)>|{\rm N}(\Delta)|$, we have $\mathcal{C}_{>\theta(B)} = \emptyset$.
 
	For $k > 1$, we introduce the three following subsets $\mathcal{S}_1$, $\mathcal{S}_2$, $\mathcal{S}_3$, and work with the inclusion:
	$$
    \mathcal{C}_{>\theta(B)} \cap \mathcal{O}^k_{K,B} \subset \mathcal{S}_1 \cup \mathcal{S}_2 \cup \mathcal{S}_3.
    $$
	 \vskip 1mm
 \noindent
$\bullet$ $\mathcal{S}_1 = \{ \ul m \in \mathcal{O}^k_{K,B} \mid Q_\ell(\ul m) = 0 \}$. By the Schwartz--Zippel lemma, and since the size of $\mathcal{O}_{K,B}$ tends to $\infty$ with $B$, we have 
        $$\frac{\#\mathcal{S}_1}{\#\mathcal{O}^k_{K,B}} \rightarrow 0\,\,\,\,\textrm{as}\,\,B \to +\infty.
        $$	
        \vskip 1mm
        \noindent
		$\bullet$ $\mathcal{S}_2 = \{ \ul m \in \mathcal{O}^k_{K,B} \mid \exists \mathfrak{p} \in {\rm Spec}(\mathcal{O}_K)\setminus (\Gamma_{\leq \theta(B)} \cup \mathscr{C})\textrm{ such that } \mathfrak{p}|Q_1^0(\ul m),\ldots,\mathfrak{p} | Q_{\ell-1}^0(\ul m), \mathfrak{p} | Q_\ell(\ul m)  \}$. By the induction hypothesis, we have again
        $$\frac{\#\mathcal{S}_2}{\#\mathcal{O}^k_{K,B}}\rightarrow 0\,\,\,\,\textrm{as}\,\,B \to +\infty.
        $$ 
	\vskip 1mm
    \noindent
    $\bullet$ $\mathcal{S}_3 = \{ \ul m \in \mathcal{O}^k_{K,B} \mid Q_\ell(\ul m) \neq 0, \exists \mathfrak{p} \in {\rm Spec}(\mathcal{O}_K)\setminus (\Gamma_{\leq \theta(B)} \cup \mathscr{C}),  \mathfrak{p}|Q_1(\ul m),\ldots, \mathfrak{p} | Q_\ell(\ul m),  \mathfrak{p} \nmid Q_{i_0}^0(\ul m) \text{ for some } 1 \le i_0 \leq \ell-1 \}$.
        As in Lemma \ref{lem:funfunc}, for sufficiently large $B$, and for all $\underline{m} \in \mathcal{O}_{K,B}^k$ with $Q_\ell(\underline{m}) \neq 0$, there are at most {$L(B):=\frac{\gamma_{Q_\ell,\fd} \log_q(B) }{\log_q(\theta(B))}$} prime ideals $\mathfrak{p} \notin (\Gamma_{\leq \theta(B)} \cup \mathscr{C})$ such that $\mathfrak{p}|Q_\ell(\ul m)$. In the remainder, we assume that $B$ is large enough to satisfy this condition.	
  
		Fix a tuple $(m_1,\ldots,m_{k-1}) \in \mathcal{O}^{k-1}_{K,B}$, and consider $\ul m= (m_1,\ldots,m_{k-1},m_k) \in \mathcal{S}_3$. 
		
		Let $\mathfrak{p}$ be a prime ideal such that, for some index $i_0\in \{1,\ldots,\ell-1\}$, we have $\mathfrak{p} \nmid Q_{i_0}^0(m_1,\ldots,m_{k-1})$. Write 
		$$Q_{i_0}(m_1,\ldots,m_{k-1},x)
		= Q_{i_0}^0(m_1,\ldots,m_{k-1}) x^{\delta_{i_0}} + \cdots$$
		Let $D_B$ be as in Corollary~\ref{lem:boundinarithproBC}. Then, there are at most $\delta_{i_0} D_B$ elements $x=m_k \in \mathcal{O}_K$ with $0 \le {\rm H}(m_k) \leq B$ such that $\mathfrak{p} | Q_{i_0}(m_1,\ldots,m_{k-1},m_k)$.
		Therefore, for each $(m_1,\ldots,m_{k-1}) \in \mathcal{O}^{k-1}_{K,B}$, there are at most $$D'_B=D_B(\ell-1) \cdot L(B) \cdot \delta_1 \cdots \delta_{\ell-1}
        $$ values of $m_k \in \mathcal{O}_{K,B}$ such that $(m_1,\ldots,m_{k-1},m_k) \in \mathcal{S}_3$.
		Hence, we conclude that 
        $$
        \frac{\#\mathcal{S}_3}{\#\mathcal{O}^k_{K,B}} \le \frac{D'_B \cdot \#\mathcal{O}^{k-1}_{K,B}}{\#\mathcal{O}^{k}_{K,B}} = \frac{D'_B}{\#\mathcal{O}_{K,B}}\rightarrow 0\,\,\,\,\textrm{as}\,\,B \to +\infty.
        $$ 
	
\subsection{Conclusion}
	Recall that $\mathcal{C}_{>M} = \mathcal{C}_{>M, \le \theta(B)} \cup \mathcal{C}_{>\theta(B)}$. Therefore, by \eqref{eq:QMB}, we have:
	$$\frac{\#(\mathcal{C}_{>M} \cap \mathcal{O}^k_{K,B}) }{\# \mathcal{O}^k_{K,B}} \le 
	\frac{\# (\mathcal{C}_{>M, \le \theta(B)} \cap \mathcal{O}^k_{K,B})}{\# \mathcal{O}^k_{K,B}} 
	+ \frac{\# (\mathcal{C}_{>\theta(B)} \cap \mathcal{O}^k_{K,B})}{\# \mathcal{O}^k_{K,B}}
	\le C \sum_{{\rm N}(\mathfrak{p}) > M} \frac{1}{{\rm N}(\mathfrak{p})^2} + \frac{\# (\mathcal{C}_{>\theta(B)} \cap \mathcal{O}^k_{K,B})}{\# \mathcal{O}^k_{K,B}}.$$
	By \eqref{eq:QB}, the last term tends to $0$ as $B \to +\infty$,
	and the first term tends to $0$ as $M  \to +\infty$. Hence, we conclude 
    $$
    \mu_{\rm sup}(\mathcal{C}_{>M}) \rightarrow 0\,\,\,\,\textrm{as}\,\,M \rightarrow \infty.
    $$
\vskip 1mm
\noindent
{\bf Acknowledgements:} The first author is supported by a JSPS Postdoctoral Fellowship, and also gratefully acknowledges Noy Soffer Aranov, Christian Elsholtz and TU Graz for funding his research visits at TU Graz between March and June 2026. The second author acknowledges the support of the CDP C2EMPI, together with the French State under the France-2030 programme, the University of Lille, the Initiative of Excellence of the University of Lille, the European Metropolis of Lille for their funding and support of the R-CDP-24-004-C2EMPI project.

\bibliography{attempt}
\bibliographystyle{alpha}
\end{document}